\documentclass[12pt]{amsart}

\usepackage[utf8]{inputenc}
\usepackage[T1,T2A]{fontenc}
\usepackage[english]{babel}
\usepackage{amsfonts}
\usepackage{amsbsy}
\usepackage{amssymb}
\usepackage{fixltx2e,graphicx}
\usepackage{mathrsfs}
\usepackage{hyperref}
\usepackage{longtable}
\usepackage{enumitem}
\usepackage{comment}

\catcode`~=11 % make LaTeX treat tilde (~) like a normal character
\newcommand{\urltilde}{\kern -.15em\lower .7ex\hbox{~}\kern .04em}
\catcode`~=13 % revert back to treating tilde (~) as an active character

\makeatletter

\renewcommand{\@makecaption}[2]{
\vspace{\abovecaptionskip}%
\sbox{\@tempboxa}{#1. #2}%
\global\@minipagefalse \hbox to \hsize {{\scshape \hfil #1.
#2\hfil}} \vspace{\belowcaptionskip}}

\makeatother

\newcommand{\Hom}{\operatorname{Hom}}
\newcommand{\rk}{\operatorname{rk}}
\newcommand{\Supp}{\operatorname{Supp}}
\newcommand{\supp}{\operatorname{supp}}
\newcommand{\Ker}{\operatorname{Ker}}

\newcommand{\GL}{\operatorname{GL}}
\newcommand{\SL}{\operatorname{SL}}

\newcommand{\Spin}{\operatorname{Spin}}
\newcommand{\SO}{\operatorname{SO}}

\newcommand{\diag}{\operatorname{diag}}
\renewcommand{\div}{\operatorname{div}}
\newcommand{\ord}{\operatorname{ord}}

\newcommand{\ZZ}{\mathbb Z}
\newcommand{\QQ}{\mathbb Q}
\newcommand{\CC}{\mathbb C}

\newtheorem{theorem}{Theorem}

\newtheorem{proposition}[theorem]{Proposition}
\newtheorem{lemma}[theorem]{Lemma}
\newtheorem{corollary}[theorem]{Corollary}

\newtheorem*{question*}{Question}

\theoremstyle{definition}

\newtheorem{example}[theorem]{Example}

\theoremstyle{remark}

\newtheorem{remark}[theorem]{Remark}

\makeatletter

\@addtoreset{theorem}{section}

\makeatother

\numberwithin{equation}{section}

\numberwithin{equation}{section}

\newcounter{num}
\newcommand{\newstep}{\refstepcounter{num}\arabic{num}}

\begin{document}

\renewcommand{\proofname}{Proof}
\renewcommand{\abstractname}{Abstract}
\renewcommand{\refname}{References}
\renewcommand{\figurename}{Figure}
\renewcommand{\tablename}{Table}

\title%[On extended weight monoids]
{On extended weight monoids of spherical homogeneous spaces}

\author{Roman Avdeev}

%\thanks{}
\dedicatory{To the memory of my teacher Ernest Borisovich Vinberg}

\address{%
{\bf Roman Avdeev} \newline\indent National Research University ``Higher School of Economics'', Moscow, Russia}

\email{suselr@yandex.ru}

%\date{\today}

\subjclass[2010]{14M27, 14M17, 20G05}

\keywords{Algebraic group, spherical variety, spherical subgroup, weight monoid, extended weight monoid}

\begin{abstract}
Given a connected reductive complex algebraic group~$G$ and a spherical subgroup~$H \subset G$, the extended weight monoid $\widehat \Lambda^+_G(G/H)$ encodes the $G$-module structures on spaces of global sections of all $G$-linearized line bundles on~$G/H$.
Assuming that $G$ is semisimple and simply connected and $H$ is specified by a regular embedding in a parabolic subgroup $P \subset G$, in this paper we obtain a description of $\widehat \Lambda^+_G(G/H)$ via the set of simple spherical roots of~$G/H$ together with certain combinatorial data explicitly computed from the pair~$(P,H)$.
As an application, we deduce a new proof of a result of Avdeev and Gorfinkel describing $\widehat \Lambda^+_G(G/H)$ in the case where $H$ is strongly solvable.
\end{abstract}

\maketitle

\section{Introduction}

Throughout this paper, we work over the field $\CC$ of complex numbers.

Let $G$ be a connected reductive algebraic group and let $B$ be a Borel subgroup of~$G$.
A $G$-variety (that is, an algebraic variety equipped with a regular action of~$G$) is said to be \textit{spherical} if it is irreducible, normal, and possesses an open $B$-orbit.
In particular, one may speak of spherical homogeneous spaces and spherical (finite-dimensional) $G$-modules.
A closed subgroup $H \subset G$ is said to be \textit{spherical} if $G/H$ is a spherical homogeneous space.

Let $X$ be a normal irreducible $G$-variety.
According to a result of Vinberg and Kimelfeld \cite[Theorem~2]{VK78}, $X$ being spherical implies that the algebra $\CC[X]$ of regular functions on~$X$ is multiplicity free as a $G$-module, and the converse is also true if $X$ is quasi-affine.
The highest weights of all simple $G$-modules that occur in~$\CC[X]$ form a monoid, called the \textit{weight monoid} of~$X$; we denote it by~$\Lambda^+_G(X)$.
If $X$ is spherical then $\Lambda^+_G(X)$ uniquely determines the $G$-module structure on~$\CC[X]$ and hence is an important invariant of~$X$.

Now suppose $X = G/H$ is a homogeneous space.
Then another result of Vinberg and Kimelfeld~\cite[Theorem~1]{VK78} asserts that $G/H$ being spherical is equivalent to the following property: for every $G$-linearized line bundle $\mathcal L$ on~$G/H$, the space $\Gamma(\mathcal L)$ of global sections of~$\mathcal L$ is multiplicity free as a $G$-module.
Thanks to the Frobenius reciprocity theorem (see~\cite[\S\,2]{VK78} or~\cite[Corollary~2.13]{Tim}), the latter property can be reformulated as follows: for every simple finite-dimensional $G$-module~$V$ and every character $\chi$ of~$H$, the subspace $V^{(H)}_\chi$ of $H$-semiinvariant vectors in~$V$ of weight~$\chi$ is at most one-dimensional (see details in~\S\S\,\ref{subsec_LLB},\,\ref{subsec_EWM}).

Let $\mathfrak X(B), \mathfrak X(H)$ be the character groups of $B,H$ (in additive notation), respectively, and let $\Lambda^+_G \subset \mathfrak X(B)$ be the monoid of dominant weights of~$G$ with respect to~$B$.
For every $\lambda \in \Lambda^+_G$, denote by~$V_G(\lambda)$ the simple $G$-module with highest weight~$\lambda$ and let $V_G(\lambda)^*$ be the respective dual $G$-module.
Then the set
\[
\widehat \Lambda^+_G(G/H) = \lbrace (\lambda, \chi) \in \Lambda^+_G \times \mathfrak X(H) \mid [V_G(\lambda)^*]^{(H)}_\chi \ne 0 \rbrace
\]
is a monoid, called the \textit{extended weight monoid} (or \textit{extended weight semigroup}) of~$G/H$.
The word ``extended'' is justified by the fact that the map $\lambda \mapsto (\lambda, 0)$ identifies $\Lambda^+_G(G/H)$ with the submonoid $\lbrace (\lambda,\chi) \mid \chi = 0 \rbrace \subset \widehat \Lambda^+_G(G/H)$.
In particular, $\Lambda^+_G(G/H) \simeq \widehat \Lambda^+_G(G/H)$ whenever $\mathfrak X(H)$ is trivial.
If $G/H$ is spherical then $\widehat \Lambda^+_G(G/H)$ uniquely determines the $G$-module structures on spaces of global sections of all $G$-linearized line bundles on~$G/H$, therefore a natural problem is to find explicit formulas and/or algorithms for computing the extended weight monoids for spherical homogeneous spaces.

Another motivation for the above-mentioned problem comes from the general theory of spherical varieties.
Given a spherical homogeneous space $G/H$, one defines three main combinatorial invariants: the \textit{weight lattice} $\Lambda_G(G/H) \subset \mathfrak X(B)$, the finite set $\Sigma_G(G/H) \subset \Lambda_G(G/H)$ of \textit{spherical roots}, and the finite set $\mathcal D_G(G/H)$ of $B$-stable prime divisors in~$G/H$, called \textit{colors}, which is regarded as an abstract set equipped with a map to the dual lattice of $\Lambda_G(G/H)$ (see the precise definitions in~\S\,\ref{subsec_spherical_varieties}).
These invariants play a crucial role in the combinatorial classification of spherical homogeneous spaces (see~\cite[\S\,30.11]{Tim},~\cite{BP16}, and references therein) and their equivariant embeddings (see~\cite{LV} or~\cite{Kn91}), therefore an important problem is to find effective methods for computing them.
Now the significance of the extended weight monoid $\widehat \Lambda^+_G(G/H)$ is demonstrated by the fact that the invariants $\Lambda_G(G/H)$ and $\mathcal D_G(G/H)$ can be expressed purely in terms of $\widehat \Lambda^+_G(G/H)$ and moreover the pair $(\Sigma_G(G/H), \widehat \Lambda^+_G(G/H))$ uniquely determines $H$ up to conjugacy; see~\cite[\S\,2.3]{Avd_solv_inv} for details.

An interesting application of the extended weight monoids of spherical homogeneous spaces to matrix-valued orthogonal polynomials can be found in~\cite{PeVP}.

For a general spherical homogeneous space $G/H$, computing the monoid $\widehat \Lambda^+_G(G/H)$ easily reduces to the case where $G$ is semisimple and simply connected (see Remark~\ref{rem_relax}).
In the latter case it is well known that $\widehat \Lambda^+_G(G/H)$ is naturally isomorphic to the monoid of $B$-stable effective divisors on~$G/H$ (see Theorem~\ref{thm_ewm_is_free}); in particular, $\widehat \Lambda^+_G(G/H)$ is free and its indecomposable elements are in bijection with the set $\mathcal D_G(G/H)$ of colors.
More specifically, given a color $D \in \mathcal D_G(G/H)$, the corresponding indecomposable element of $\widehat \Lambda^+_G(G/H)$ is precisely the $(B \times H)$-biweight of a (unique up to proportionality) regular function on~$G$ whose divisor of zeros is the pullback of~$D$ via the map $G \to G/H$.

In this paper, we are concerned with the problem of computing the monoid $\Lambda^+_G(G/H)$ for a spherical subgroup $H$ of a simply connected semisimple group~$G$ in terms of a regular embedding of $H$ in a parabolic subgroup $P \subset G$, where ``regular'' means the inclusion $H_u \subset P_u$ of the unipotent radicals of $H$ and~$P$, respectively.
In this setting, a complete description of $\widehat \Lambda^+_G(G/H)$ is known essentially in the following particular cases:
\begin{enumerate}[label=\textup{(C\arabic*)},ref=\textup{C\arabic*}]
\item \label{case_P=G}
$P = G$ (that is, $H$ is reductive, in which case $G/H$ is affine);

\item \label{case_P=B}
$P = B$ (subgroups contained in a Borel subgroup of $G$ are called \textit{strongly solvable}).
\end{enumerate}
In case~(\ref{case_P=G}), there is a complete classification of spherical homogeneous spaces $G/H$ with reductive $H$ due to \cite{Kr, Mi, Br87}, and the monoids $\widehat \Lambda^+_G(G/H)$ (or in some cases the closely related monoids~$\Lambda_G^+(G/H)$) are known for them thanks to~\cite{Kr, Avd_EWS}.
In case~(\ref{case_P=B}), the monoids $\widehat \Lambda^+_G(G/H)$ were described in~\cite{AvdG} via a structure theory of strongly solvable spherical subgroups developed in~\cite{Avd_solv}.

Coming back to an arbitrary spherical subgroup $H \subset G$ regularly embedded in a parabolic subgroup $P \subset G$, we now outline the approach for computing $\widehat \Lambda^+_G(G/H)$ used in this paper.
Choose Levi subgroups $K \subset H$, $L \subset P$ such that $K \subset L$ and recall from the above discussion that the indecomposable elements of $\widehat \Lambda^+_G(G/H)$ are in bijection with the colors of~$G/H$.
Consider the chain of fibrations $G/H \to G/KP_u \to G/P$ with irreducible fibers.
Then the colors of~$G/H$ naturally fall into the following three types:
\begin{enumerate}[label=\textup{(\arabic*)},ref=\textup{\arabic*}]
\item \label{type1}
pullbacks of colors of $G/P$;

\item \label{type2}
pullbacks of colors of $G/KP_u$ that are not of type~(\ref{type1});

\item \label{type3}
all the remaining colors.
\end{enumerate}
The colors of type~(\ref{type1}) are well understood and the corresponding $(B \times H)$-biweights are easy to write down.
The colors of type~(\ref{type2}) are in a natural bijection with the colors of~$L/K$ and the corresponding $(B \times H)$-biweights are not hard to determine using the above-mentioned classification results on affine spherical homogeneous spaces.
The colors of type~(\ref{type3}) are more elusive and our main result (see Theorem~\ref{thm_gen_3}) expresses their $(B\times H)$-biweights via those of colors of types~(\ref{type1}) and~(\ref{type2}), known classification results on spherical modules (see Remark~\ref{rem_SM}), and the set $\Pi \cap \Sigma_G(G/H)$ where $\Pi$ is the set of simple roots of~$G$.

As can be seen, the main obstacle for using our results in practical calculations is the problem of determining the set $\Pi \cap \Sigma_G(G/H)$ of simple spherical roots of~$G/H$.
In turn, a certain progress in computing the whole set $\Sigma_G(G/H)$ has been made in~\cite{Avd_degen}.
Namely, loc. cit. proposes effective algorithms for computing $\Sigma_G(G/H)$ in the case $L' \subset K \subset\nobreak L$ (where $L'$ is the derived subgroup of~$L$) along with a number of results that can help to determine $\Sigma_G(G/H)$ for a concrete spherical subgroup~$H$ whenever one is lucky to construct degenerations of the Lie algebra of~$H$ with special properties.
In our Example~\ref{ex_SL6}, we demonstrate a concrete realization of the latter idea by computing $\Sigma_G(G/H)$ for a spherical subgroup $H$ with $L' \not\subset K$.
In any case, combining our results with those of loc. cit. we get an effective procedure for computing $\widehat \Lambda^+_G(G/H)$ for spherical subgroups $H$ satisfying $L' \subset K \subset L$; see Example~\ref{ex_SO7} for a demonstration.

The discussion in the previous paragraph suggests that it is still an important problem to find a way to compute the set $\Pi \cap \Sigma_G(G/H)$ directly from $P$ and~$H$.
For a given $\alpha \in\nobreak \Pi$, we propose in Propositions~\ref{prop_alpha_in_Sigma1} and~\ref{prop_alpha_in_Sigma2} several necessary and one sufficient condition for $\alpha$ to lie in~$\Sigma_G(G/H)$.
Fortunately, in the case where $H$ is strongly solvable (so that $P$ is a Borel subgroup) the sufficient condition is automatically fulfilled for all~$\alpha$ satisfying all the necessary conditions, and this enables us to deduce a new proof of the main result of~\cite{AvdG} on the description of~$\widehat \Lambda^+_G(G/H)$ for such~$H$.
We point out that the proof in loc. cit. was based on explicit constructions of $H$-semiinvariant vectors in certain simple $G$-modules whereas our new proof avoids them.
Unfortunately, in the general case our necessary/sufficient conditions are definitely not enough to compute the set $\Pi \cap \Sigma_G(G/H)$: Example~\ref{ex_SO7} exhibits a simple root $\alpha \in \Pi \setminus \Sigma_G(G/H)$ satisfying all our necessary conditions.
Thus computing $\Pi \cap \Sigma_G(G/H)$ directly remains an open problem, and we discuss in~\S\,\ref{sect_concluding_remarks} possible directions of further research related to it.

This paper is organized as follows.
In~\S\,\ref{sect_notation} we set up notation and conventions.
In~\S\,\ref{sect_preliminaries} we collect all the material needed to state our main results.
In~\S\,\ref{sect_main_results} we state the main results of this paper.
In~\S\,\ref{sect_tools} we gather several auxiliary results needed in our proofs.
The main results are proved in~\S\,\ref{sect_proofs}.
In~\S\,\ref{sect_strongly_solvable} we provide a new proof of the main result of~\cite{AvdG}.
In~\S\,\ref{sect_examples} we present several examples of computing~$\widehat \Lambda^+_G(G/H)$.
In~\S\,\ref{sect_concluding_remarks} we discuss the problem of computing the set~$\Pi \cap \Sigma_G(G/H)$ directly.

\subsection*{Acknowledgements}

A part of this work was done while the author was visiting the Institut Fourier in Grenoble, France, in October~2019.
He thanks this institution for hospitality and excellent working conditions and also expresses his gratitude to Michel Brion for support, valuable comments on this work, and especially for communicating the proof of Proposition~\ref{prop_pullback}(\ref{prop_pullback_a}).
Thanks are also due to Dmitry Timashev and Vladimir Zhgoon for useful discussions as well as to the referee for helpful comments and suggestions.

This work was supported by the Russian Science Foundation, grant no.~18-71-00115.

\section{Notation and conventions}
\label{sect_notation}

Throughout this paper, all topological terms relate to the Zariski topology, all groups are assumed to be algebraic unless they explicitly appear as character groups, all subgroups of algebraic groups are assumed to be algebraic.
The Lie algebras of algebraic groups denoted by capital Latin letters are denoted by the corresponding small Gothic letters.
Given a group~$K$, a~\textit{$K$-variety} is an algebraic variety equipped with a regular action of~$K$.

$\ZZ^+ = \lbrace z \in \ZZ \mid z \ge 0 \rbrace$;

$|X|$ is the cardinality of a finite set~$X$;

$\CC^\times$ is the multiplicative group of the field~$\CC$;

$e$ is the identity element of any group;

$V^*$ is the vector space of linear functions on a vector space~$V$;

$L^\vee = \Hom_\ZZ(L,\ZZ)$ is the dual lattice of a lattice~$L$;

$\mathfrak X(K)$ is the character group (in additive notation) of a group~$K$;

$K^0$ is the connected component of the identity of an algebraic group~$K$;

$K'$ is the derived subgroup of a group~$K$;

$K_u$ is the unipotent radical of an algebraic group~$K$;

$K_x$ is the stabilizer of a point $x$ of a $K$-variety;

$Z(K)$ is the center of a group~$K$;

$\CC[X]$ is the algebra of regular functions on an algebraic variety~$X$;

$\CC(X)$ is the field of rational functions on an irreducible algebraic variety~$X$;

$\Gamma(\mathcal L)$ is the space of global sections of a line bundle $\mathcal L$ on a smooth algebraic variety;

$V^{(K)}_\chi$ is the space of semi-invariants of weight $\chi \in \mathfrak X(K)$ for an action of a group $K$ on a vector space~$V$;

$T_x X$ is the tangent space of an algebraic variety~$X$ at a point~$x \in X$;

$A^T$ is the transpose of a matrix~$A$;

$G$ is a connected reductive algebraic group;

$B \subset G$ is a fixed Borel subgroup;

$T \subset B$ is a fixed maximal torus;

$B^-$ is the Borel subgroup of~$G$ opposite to~$B$ with respect to~$T$, so that $B \cap B^- = T$;

$(\cdot\,,\,\cdot)$ is a fixed inner product on~$\QQ\mathfrak X(T)$ invariant with respect to the Weyl group $N_G(T)/T$;

$\Delta \subset \mathfrak X(T)$ is the root system of~$G$ with respect to~$T$;

$\Delta^+ \subset \Delta$ is the set of positive roots with respect to~$B$;

$\Pi \subset \Delta^+$ is the set of simple roots with respect to~$B$;

$\mathfrak g_\alpha \subset \mathfrak g$ is the root subspace corresponding to a root~$\alpha \in \Delta$;

$\Lambda_G^+ \subset \mathfrak X(T)$ is the set of dominant weights of $G$ with respect to~$B$;

$V_G(\lambda)$ is a simple $G$-module with highest weight $\lambda \in \Lambda_G^+$.

The groups $\mathfrak X(B)$ and $\mathfrak X(T)$ are identified via restricting characters from~$B$ to~$T$.

For every $\beta = \sum \limits_{\alpha \in \Pi} k_\alpha \alpha \in \Delta^+$, we put $\Supp \beta = \lbrace \alpha \in \Pi \mid k_\alpha > 0 \rbrace$.

When $G$ is semisimple, for every $\alpha \in \Pi$ we let $\varpi_\alpha \in \mathfrak X(T) \otimes_\ZZ \QQ$ be the corresponding fundamental weight.
Then every $\lambda \in \Lambda^+_G$ is uniquely expressed as $\lambda = \sum \limits_{\alpha \in \Pi} c_\alpha \varpi_\alpha$ with $c_\alpha \in \ZZ^+$, and we put $\supp \lambda = \lbrace \varpi_\alpha \mid c_\alpha > 0 \rbrace$.

Given a parabolic subgroup $P \subset G$ such that $P \supset B$ or $P \supset B^-$, we identify $\mathfrak X(P)$ with a sublattice of~$\mathfrak X(T)$ via restricting characters from~$P$ to~$T$.
The unique Levi subgroup $L$ of $P$ containing~$T$ is called the \textit{standard Levi subgroup} of~$P$.
Let $\Pi_L \subset \Pi$ be the set of simple roots of~$L$.

\section{Preliminaries}
\label{sect_preliminaries}

\subsection{Spherical varieties}
\label{subsec_spherical_varieties}

Recall from the introduction that a $G$-variety $X$ is said to be \textit{spherical} if it is irreducible, normal, and possesses an open $B$-orbit.
A subgroup $H \subset G$ is said to be \textit{spherical} if $G/H$ is a spherical homogeneous space.

In this subsection, we define several combinatorial invariants of spherical varieties that will be needed in our paper.

Let $X$ be a spherical $G$-variety.

By definition, the \textit{weight lattice} of $X$ is
$\Lambda_G(X) = \lbrace \lambda \in \mathfrak X(T) \mid \CC(X)_\lambda^{(B)} \ne \lbrace 0 \rbrace \rbrace$.
Since $B$ has an open orbit in~$X$, it follows that for every $\lambda \in \Lambda_G(X)$ the space $\CC(X)^{(B)}_\lambda$ has dimension~$1$ and hence is spanned by a nonzero function~$f_\lambda$.

Put $\mathcal Q_G(X) = \Hom_\ZZ(\Lambda_G(X), \QQ)$.

Every discrete $\QQ$-valued valuation $v$ of the field $\CC(X)$ vanishing on $\CC^\times$ determines an element $\rho_v \in \mathcal Q_G(X)$ such that $\rho_v(\lambda) = v(f_\lambda)$ for all $\lambda \in \Lambda_G(X)$.
It is known that the restriction of the map $v \mapsto \rho_v$ to the set of $G$-invariant discrete $\QQ$-valued valuations of $\CC(X)$ vanishing on~$\CC^\times$ is injective (see~\cite[7.4]{LV} or~\cite[Corollary~1.8]{Kn91}) and its image is a finitely generated cone containing the image in~$\mathcal Q_G(X)$ of the antidominant Weyl chamber (see \cite[Proposition~3.2 and Corollary~4.1,~i)]{BriP} or~\cite[Corollary~5.3]{Kn91}).
We denote this cone by~$\mathcal V_G(X)$; it is called the \textit{valuation cone} of~$X$.
Results of~\cite[\S\,3]{Bri90} imply that $\mathcal V_G(X)$ is a cosimplicial cone in~$\mathcal Q_G(X)$.
Consequently, there is a uniquely determined linearly independent set $\Sigma_G(X)$ of primitive elements in~$\Lambda_G(X)$ such that
\[
\mathcal V_G(X) = \lbrace q \in \mathcal Q_G(X) \mid q(\sigma) \le 0 \ \text{for all} \ \sigma \in \Sigma_G(X) \rbrace.
\]
Elements of $\Sigma_G(X)$ are called \textit{spherical roots} of~$X$.
The above discussion implies that every spherical root is a nonnegative linear combination of simple roots.
It also follows from loc.~cit. that $(\sigma_1,\sigma_2) \le 0$ for any two distinct spherical roots $\sigma_1,\sigma_2 \in \Sigma_G(X)$.

We mention the following well-known result.

\begin{proposition} \label{prop_two_SS}
Suppose $H_1,H_2 \subset G$ are two spherical subgroups such that $H_1 \subset H_2$.
Then
\begin{enumerate}[label=\textup{(\alph*)},ref=\textup{\alph*}]
\item \label{prop_two_SS_a}
$\Lambda_G(G/H_2) \subset \Lambda_G(G/H_1)$;
\item \label{prop_two_SS_b}
each $\sigma \in \Sigma_G(G/H_2)$ is a nonnegative linear combination of elements in~$\Sigma_G(G/H_1)$.
\end{enumerate}
\end{proposition}

\begin{proof}
The morphism $G/H_1 \to G/H_2$ yields the inclusion in~(\ref{prop_two_SS_a}) and a map $\mathcal V_G(G/H_1) \to \mathcal V_G(G/H_2)$.
By~\cite[Corollary~1.5]{Kn91} the latter map is surjective, whence~(\ref{prop_two_SS_b}).
\end{proof}

Given a $B$-stable prime divisor $D \subset X$, let $v_D$ be the valuation of the field $\CC(X)$ defined by~$D$, that is, $v_D(f) = \ord_D(f)$ for every $f \in \CC(X)$.
Observe that $\rho_{v_D}$ is actually contained in the lattice $\Lambda_G(X)^\vee \subset \mathcal Q_G(X)$.

A \textit{color} of $X$ is a prime divisor in~$X$ that is $B$-stable but not $G$-stable.
Let $\mathcal D_G(X)$ denote the (finite) set of all colors of~$X$.
The set $\mathcal D_G(X)$ is considered as an abstract set equipped with the map $\rho \colon \mathcal D_G(X) \to \Lambda_G(X)^\vee$ given by $D \mapsto \rho_{v_D}$.

Note that $\Lambda_G(X)$, $\Sigma_G(X)$, and $\mathcal D_G(X)$ depend only on the open $G$-orbit in~$X$.

The \textit{weight monoid} of~$X$ is $\Lambda^+_G(X) = \lbrace \lambda \in \Lambda^+_G \mid \CC[X]^{(B)}_\lambda \ne \lbrace 0 \rbrace \rbrace$.
Clearly, $\Lambda^+_G(X) \subset \Lambda_G(X)$.
The following result is also well known; for a proof see, for instance,~\cite[Proposition~5.14]{Tim}.

\begin{proposition} \label{prop_SM_WL_and_WM}
If $X$ is quasi-affine then $\Lambda_G(X) = \ZZ \Lambda^+_G(X)$.
\end{proposition}

\subsection{Spherical modules}
\label{subsec_spherical_modules}

A finite-dimensional $G$-module $V$ is said to be \textit{spherical} if $V$ is spherical as a $G$-variety.

Let $V$ be a spherical $G$-module and for every $\lambda \in \Lambda^+_G(V)$ fix a nonzero element $f_\lambda$ in the (one-dimensional) subspace $\CC[X]^{(B)}_\lambda$.
Clearly, for all $\lambda,\mu \in \Lambda^+_G(V)$ the product $f_\lambda f_\mu$ is proportional to~$f_{\lambda+\mu}$.
Let $\mathcal D_B$ (resp.~$\ZZ^+\mathcal D_B$) denote the set of $B$-stable prime (resp. effective) divisors in~$V$.
A combination of \cite[Theorem~3.1]{PV} with the unique factorization property of~$\CC[V]$ yields the following result (compare with~\cite[Theorem~3.2]{Kn98}).

\begin{theorem} \label{thm_wm_of_sm}
The map $\Lambda^+_G(V) \to \ZZ^+\mathcal D_B$, $\lambda \mapsto \div f_\lambda$, is an isomorphism.
In particular, $\Lambda^+_G(V)$ is a free monoid of rank~$|\mathcal D_B|$.
\end{theorem}

\begin{remark} \label{rem_SM}
There is a complete classification of spherical modules due to~\cite{Kac,BR,Lea}.
The weight monoids for all classified cases were computed in~\cite{HU, Lea} (see also~\cite{Kn98} for a convenient presentation).
\end{remark}

Now suppose $V$ is equipped with a linear action of a bigger group
$\widetilde G \supset G$ such that $\widetilde G = AG$ for a finite subgroup $A \subset Z(\widetilde G)$ (so that $\widetilde G^0 = G$).
By abuse of terminology, in this situation we say that $V$ is a spherical $\widetilde G$-module.
Put $\widetilde B = AB$ and observe that every $B$-semiinvariant rational function on~$V$ is actually $\widetilde B$-semiinvariant; thus we get well-defined notions of the weight lattice $\Lambda_{\widetilde G}(V)$ and weight monoid $\Lambda^+_{\widetilde G}(V)$ (with respect to~$\widetilde B$).
This discussion implies also the following remark, which is stated separately for future reference.

\begin{remark} \label{rem_sm_B-stable}
Every $D \in \mathcal D_B$ is actually $\widetilde B$-stable.
In particular, Theorem~\ref{thm_wm_of_sm} remains valid for spherical $\widetilde G$-modules and $\Lambda^+_{\widetilde G}(V)$ is also free.
\end{remark}

\subsection{Linearized line bundles on homogeneous spaces}
\label{subsec_LLB}

Let $N$ be a connected linear group and consider a homogeneous space $X = N/H$.
Let $\pi \colon \mathcal L \to X$ be a line bundle on~$X$.

An \textit{$N$-linearization} of~$\mathcal L$ is an $N$-action on~$\mathcal L$ such that the map $\pi$ is equivariant and for every $g \in N$, $x \in X$ the map $\pi^{-1}(x) \to \pi^{-1}(gx)$ induced by the action of~$g$ is linear.
If $\mathcal L$ is equipped with an $N$-linearization then it is called an \textit{$N$-linearized line bundle} on~$X$.

According to~\cite[Theorem~4]{Pop}, the ($N$-equi\-variant isomorphism classes of) $N$-linearized line bundles on $N/H$ are in bijection with~$\mathfrak X(H)$.
For any $\chi \in \mathfrak X(H)$, let $\CC_\chi$ be a one-dimensional $H$-module on which $H$ acts via~$\chi$.
Then the $N$-linearized line bundle corresponding to~$\chi$ is
$N *_H \CC_\chi = (N \times \CC_\chi)/H$ where the quotient is taken for the action of~$H$ given by $(h,(g,v)) \mapsto (gh^{-1},\chi(h)v)$.
For every $g \in N$ and $v \in \CC_\chi$, let $[g,v]$ denote the image of the pair $(g,v)$ in $N *_H \CC_\chi$.
Then the projection $N *_H \CC_\chi \to N/H$ is defined by $[g,v] \mapsto gH$ and the $N$-linearization is given by left translations of the first component.
Note that the fiber of $N *_H \CC_\chi$ over $eH$ is~$\CC_\chi$ (as an $H$-module).
Finally, there is an $N$-module isomorphism
\begin{equation} \label{eqn_iso_sections}
\Gamma(N *_H \CC_\chi) \simeq \CC[N]^{(H)}_{-\chi}
\end{equation}
where the $H$-semiinvariants in the right-hand side are taken with respect to the action of~$H$ on the right.

The following result is implied by~\cite[Proposition~1 and Theorem~4]{Pop} (see also \cite[Ch.~VII, Proposition~1.5]{Ray} for a closely related result).

\begin{proposition} \label{prop_unique_linearization}
If $N$ is semisimple and simply connected then every line bundle on~$N/H$ admits a unique $N$-linearization.
\end{proposition}

\subsection{The extended weight monoid}
\label{subsec_EWM}

Let $H \subset G$ be a subgroup.
Combining~(\ref{eqn_iso_sections}) with~\cite[Corollary~2.13]{Tim} we find that, for every $\lambda \in \Lambda^+_G$ and $\chi \in \mathfrak X(H)$,
\begin{equation} \label{eqn_3dims}
\dim \Gamma(G *_H \CC_{-\chi})^{(B)}_\lambda = \dim \CC[G]^{(B \times H)}_{(\lambda, \chi)} = \dim [V_G(\lambda)^*]^{(H)}_\chi
\end{equation}
where in the middle term $B$ (resp.~$H$) acts on $\CC[G]$ on the left (resp. on the right).
Let $m(\lambda,\chi)$ denote the value in~(\ref{eqn_3dims}).

By definition, the \textit{extended weight monoid} of~$G/H$ is
\[
\widehat \Lambda^+_G(G/H) = \lbrace (\lambda, \chi) \in \Lambda^+_G \times \mathfrak X(H) \mid m(\lambda,\chi) > 0 \rbrace.
\]

The next result is a particular case of~\cite[Theorem~1]{VK78}.

\begin{theorem}
The following conditions are equivalent:
\begin{enumerate}[label=\textup{(\arabic*)},ref=\textup{\arabic*}]
\item
$H$ is a spherical subgroup of~$G$.

\item
$m(\lambda, \chi) \le 1$ for all $\lambda \in \Lambda^+_G$ and $\chi \in \mathfrak X(H)$.
\end{enumerate}
\end{theorem}

Now assume that $H$ is spherical.
Then for every $(\lambda, \chi) \in \widehat \Lambda^+_G(G/H)$ there is a unique up to proportionality nonzero section $s_{\lambda,\chi} \in \Gamma(G *_H \CC_{-\chi})^{(B)}_\lambda$.
Let $\ZZ^+\mathcal D_G(G/H)$ denote the monoid of effective $B$-stable divisors on~$G/H$, which is freely generated by~$\mathcal D_G(G/H)$.
The following result is well known; see, for instance,~\cite[Lemma~6.2.2]{Lu01} or~\cite[Theorem~2]{AvdG}.

\begin{theorem} \label{thm_ewm_is_free}
Suppose that $G$ is semisimple and simply connected.
Then the map $d \colon \widehat \Lambda^+_G(G/H) \to \ZZ^+\mathcal D_G(G/H)$, $(\lambda,\chi) \mapsto \div s_{\lambda,\chi}$, is an isomorphism.
In particular, $\widehat \Lambda^+_G(G/H)$ is a free monoid of rank~$|\mathcal D_G(G/H)|$.
\end{theorem}

For every $D \in \mathcal D_G(G/H)$, we put $d^{-1}(D) = (\lambda_D,\chi_D)$.
Then the map $D \mapsto (\lambda_D,\chi_D)$ is a bijection between $\mathcal D_G(G/H)$ and the indecomposable elements of $\Lambda^+_G(G/H)$.

\begin{remark} \label{rem_relax}
The conditions on~$G$ imposed in Theorem~\ref{thm_ewm_is_free} are not very restrictive.
Indeed, for a general connected reductive group~$G$ there exists a finite covering $\varphi \colon \widetilde G \to G$ such that $\widetilde G = \widetilde C \times \widetilde G^{ss}$ where $\widetilde C$ is a torus and $\widetilde G^{ss}$ is a simply connected semisimple group.
Then for $\widetilde H = \varphi^{-1}(H)$ one has
\[
\widehat \Lambda^+_G(G/H) = \lbrace (\lambda, \chi) \in \widehat \Lambda^+_{\widetilde G}(\widetilde G/\widetilde H) \mid \lambda \in \Lambda^+_G \rbrace.
\]
In turn, if $\widetilde H^{ss}$ denotes the projection of~$\widetilde H$ to $\widetilde G^{ss}$, then the monoid $\widehat \Lambda^+_{\widetilde G}(\widetilde G/\widetilde H)$ is easily expressed from $\widehat \Lambda^+_{\widetilde G}(\widetilde G/\widetilde H^{ss})$ as described in~\cite[Proposition~2.3]{Avd_solv_inv}.
\end{remark}

\subsection{The regular embedding of a subgroup}
\label{subsec_regular_embedding}

It follows from~\cite[\S\,30.3]{Hum} that for every subgroup $H \subset G$ there exists a parabolic subgroup $P \subset G$ such that $H \subset P$ and $H_u \subset P_u$.
In this situation, we say that $H$ is \textit{regularly embedded} in~$P$.
Clearly, one can choose Levi subgroups $K \subset H$ and $L \subset P$ such that $K \subset L$, in which case we get the following diagram of inclusions:
\begin{equation} \label{eqn_regular_embedding}
\begin{array}{ccccc}
P & = & L & \rightthreetimes & P_u \\
\cup &  & \cup & & \cup\, \\
H & = & K & \rightthreetimes & H_u
\end{array}
\end{equation}

Replacing $H$ with a conjugate subgroup in~$G$ we may assume the following properties:
\begin{enumerate}[label=\textup{(P\arabic*)},ref=\textup{P\arabic*}]
\item \label{P1}
$P \supset B^-$;

\item \label{P2}
$L$ is the standard Levi subgroup of~$P$.
\end{enumerate}
Put $B_L = B \cap L$; this is a Borel subgroup of~$L$.

Thanks to the local structure theorem~\cite[Theorem~2.3, Corollary~2.4]{Kn90} (see also~\cite[Theorem~4.7]{Tim}), there exist a parabolic subgroup $Q_L \supset B_L$ of~$L$ with standard Levi subgroup~$M$ and a point in general position $x$ for the action of~$Q_L$ on $L/K$ such that the group $\widetilde S = (Q_L)_x$ satisfies
\begin{equation} \label{eqn_M'SM}
M' \subset \widetilde S \subset M
\end{equation}
(in particular, $\widetilde S$ is reductive).
We note that the group $(B_L)_x = B_L \cap \widetilde S$ satisfies $(B_L)_x = Z(\widetilde S)(B_L)_x^0$ and $(B_L)_x^0$ is a Borel subgroup of~$\widetilde S$.

Clearly, $K$ contains a subgroup $S$ conjugate to $\widetilde S$.
According to~\cite[Corollary~8.2]{Kn90} or~\cite[Theorem~3(ii)]{Pa90}, $S$ can be characterized as a stabilizer in general position for the natural action of~$K$ on $\mathfrak l / \mathfrak k$.
We note that the subgroup $S$ may be disconnected; however, property~(\ref{eqn_M'SM}) guarantees that $S = Z(S) \cdot (S^0)'$.
Then the notions of a spherical $S$-module and its weight lattice/monoid extend as explained in~\S\,\ref{subsec_spherical_modules}.

In the following theorem, the equivalence of~(\ref{thm_criterion_spherical_1}) and~(\ref{thm_criterion_spherical_2}) was proved in~\cite[Proposition~I.1]{Br87} and that of~(\ref{thm_criterion_spherical_2}) and~(\ref{thm_criterion_spherical_3}) in~\cite[Theorem~1.2]{Pa94}; see also~\cite[Theorem~9.4]{Tim}.

\begin{theorem} \label{thm_criterion_spherical}
In the setting of~\textup{(\ref{eqn_regular_embedding})}, the following conditions are equivalent:
\begin{enumerate}[label=\textup{(\arabic*)},ref=\textup{\arabic*}]
\item \label{thm_criterion_spherical_1}
$G/H$ is a spherical $G$-variety.

\item \label{thm_criterion_spherical_2}
$P/H$ is a spherical $L$-variety.

\item \label{thm_criterion_spherical_3}
$L/K$ is a spherical $L$-variety and $\mathfrak p_u/\mathfrak h_u$ is a spherical $S$-module.
\end{enumerate}
\end{theorem}

\begin{remark} \label{rem_L/K}
It is easy to see that the subgroup $K$ is spherical in~$L$ if and only if so is~$CK$.
Since $L'$ acts transitively on~$L/CK$, it follows that $L/K$ is spherical as an $L$-variety if and only if $L'/(L' \cap CK)$ is spherical as an $L'$-variety.
\end{remark}

\begin{remark}
In fact, for spherical $L/K$ the local structure theorem mentioned above enables one to determine the isomorphism type of $S$ directly from~$\Lambda_L(L/K)$.
Namely, $M$ is given by $\Pi_M = \lbrace \alpha \in \Pi \mid (\alpha, \lambda) = 0 \ \text{for all} \ \lambda \in \Lambda_L(L/K) \rbrace$ and $\widetilde S$ is determined by $\widetilde S \cap T = \bigcap \limits_{\lambda \in \Lambda_L(L/K)} \Ker \lambda$.
For all affine spherical homogeneous spaces in the lists of \cite{Kr, Mi, Br87} (see the Introduction) the Lie algebras of corresponding subgroups~$S$ can be found, for instance, in~\cite{KnVS}.
\end{remark}

\begin{remark} \label{rem_Montagard}
By~\cite[Lemma~1.4]{Mon}, there is a $K$-equivariant (and hence $S$-equivariant) isomorphism $P_u/H_u \simeq \mathfrak p_u/\mathfrak h_u$.
\end{remark}

\subsection{The weight lattice of a spherical homogeneous space}
\label{subsec_weight_lattice}

Retain the notation of~\S\,\ref{subsec_regular_embedding} and suppose that $H \subset G$ is a spherical subgroup.
Put $B_S = B_L \cap K$.

Replacing $H$ with a conjugate subgroup, we may assume the following properties:
\begin{enumerate}[label=\textup{(S\arabic*)},ref=\textup{S\arabic*}]
\item \label{S1}
$B_LK$ is open in~$L$ (which is equivalent to $\mathfrak l = \mathfrak b_L + \mathfrak k$);

\item \label{S2}
$T \cap B_S$ is a Levi subgroup of~$B_S$.
\end{enumerate}
According to the discussion in~\S\,\ref{subsec_regular_embedding}, the subgroup $S \subset K$ may be chosen to have the following additional properties:
\begin{enumerate}[label=\textup{(S\arabic*)},ref=\textup{S\arabic*}]
\setcounter{enumi}{2}
\item \label{S3}
$B_S \subset S$;

\item \label{S4}
$B_S^0$ is a Borel subgroup of~$S$;

\item \label{S5}
$S = Z(S)S^0$;

\item \label{S6}
$B_S = Z(S)B_S^0$.
\end{enumerate}

\begin{lemma} \label{lemma_S_is_unique}
Properties \textup(\ref{S3}\textup)--\textup(\ref{S6}\textup) uniquely determine~$S$ among all reductive subgroups of~$K$.
\end{lemma}

\begin{proof}
Since each connected component of~$S$ meets~$B_S$, it suffices to show that $S^0$ is the unique connected reductive subgroup of~$K$ satisfying~(\ref{S4}).
The homogeneous space $X = K/S^0$ is affine by \cite[Theorem~3.8]{Tim} and $B_S^0$ stabilizes the point $x = eS^0 \in X$.
If $B_S^0$ is a Borel subgroup of another connected reductive subgroup $R \subset K$ then the orbit $Rx \subset K$ is projective, which immediately implies $Rx =\lbrace x \rbrace$ and thus $R \subset S^0$.
Changing the roles of $S^0$ and $R$ in this argument yields the reverse inclusion, whence $R = S^0$.
\end{proof}

The proofs of Theorem~\ref{thm_criterion_spherical} in \cite[Proposition~I.1]{Br87},~\cite[Theorem~1.2]{Pa94}, or~\cite[Theorem~9.4]{Tim} actually imply the following result.

\begin{theorem} \label{thm_weight_lattice}
Suppose that $K$ and $S$ satisfy~\textup(\ref{S1}\textup)--\textup(\ref{S6}\textup) and
let $\iota \colon \mathfrak X(T) \to \mathfrak X(T \cap B_S)$ be the character restriction map.
Then
\[
\Ker \iota = \Lambda_L(L/K) \quad \text{and} \quad \Lambda_G(G/H) = \iota^{-1}(\Lambda_S(\mathfrak p_u/\mathfrak h_u))
\]
where the lattices $\Lambda_L(L/K)$ and $\Lambda_S(\mathfrak p_u/\mathfrak h_u)$ are taken with respect to $B_L$ and~$B_S$, respectively.
\end{theorem}

\section{Statement of the main results}

\label{sect_main_results}

Throughout this subsection we assume that $G$ is semisimple and simply connected.
Let $H \subset G$ be a spherical subgroup regularly embedded in a parabolic subgroup $P \subset G$.
Let $L, K, B_L, S$ be as in \S\,\ref{subsec_regular_embedding} and assume properties~(\ref{eqn_regular_embedding}), (\ref{P1}), (\ref{P2}).
We identify the groups $\mathfrak X(H)$ and~$\mathfrak X(K)$ via the restriction of characters.

For brevity, starting from this subsection and till the end of the paper we shall write $\Lambda$, $\Sigma$, $\mathcal D$, $\widehat \Lambda^+$ instead of $\Lambda_G(G/H)$, $\Sigma_G(G/H)$, $\mathcal D_G(G/H)$, $\widehat \Lambda^+_G(G/H)$, respectively.

Our main goal in this subsection is to provide a description of the monoid~$\widehat \Lambda^+$.
We know from Theorem~\ref{thm_ewm_is_free} that $\widehat \Lambda^+$ is free, therefore it suffices to determine its rank and describe its indecomposable elements.

Combining Theorem~\ref{thm_criterion_spherical} with Remark~\ref{rem_L/K}, we find that $L'/(L' \cap CK)$ is an affine spherical homogeneous space and $\mathfrak p_u/\mathfrak h_u$ is a spherical $S$-module.
Thanks to our assumptions on~$G$, the group $L'$ is semisimple and simply connected, therefore $\widehat \Lambda^+_{L'}(L'/(L'\cap CK))$ is free by Theorem~\ref{thm_ewm_is_free}.
Recall from Theorem~\ref{thm_wm_of_sm} and Remark~\ref{rem_sm_B-stable} that the monoid $\Lambda^+_S(\mathfrak p_u/\mathfrak h_u)$ is also free.

\begin{theorem} \label{thm_rank}
The following equality holds:
\begin{equation} \label{eqn_rank}
\rk \widehat \Lambda^+ = |\Pi \setminus \Pi_L| + \rk \widehat \Lambda^+_{L'}(L'/(L'\cap CK)) + \rk \Lambda^+_S(\mathfrak p_u / \mathfrak h_u).
\end{equation}
\end{theorem}

Let $\Xi$ denote the set of indecomposable elements of~$\widehat \Lambda^+$.
In what follows, $\Xi$ will be divided into three parts $\Xi_1,\Xi_2,\Xi_3$ according to the summands in the right-hand side of~(\ref{eqn_rank}), and we are going to describe these parts one by one.

For every $\alpha \in \Pi \setminus \Pi_L$, let $\overline \varpi_\alpha$ denote the restriction to $H$ of $\varpi_\alpha$ regarded as an element of~$\mathfrak X(P)$.
Put $\Xi_1 = \lbrace (\varpi_\alpha, - \overline \varpi_\alpha) \mid \alpha \in \Pi \setminus \Pi_L \rbrace$.

\begin{theorem} \label{thm_gen_1}
One has $\Xi_1 \subset \Xi$.
\end{theorem}

Next we proceed to defining the set $\Xi_2$.

The character restriction map $\mathfrak X(T) \to \mathfrak X(T \cap L')$ induces an isomorphism
\begin{equation} \label{eqn_restr}
\tau_L \colon \ZZ\lbrace \varpi_\alpha \mid \alpha \in \Pi_L \rbrace \xrightarrow{\sim} \mathfrak X(T \cap L'),
\end{equation}
and for every $\lambda \in \mathfrak X(T \cap L')$ we put $\widetilde \lambda = \tau_L^{-1}(\lambda)$.
In particular, $\widetilde \lambda = \varpi_\alpha$ whenever $\lambda$ is the fundamental weight of $L'$ corresponding to a simple root $\alpha \in \Pi_L$.

Let $\Xi'_2$ denote the set of indecomposable elements of $\widehat \Lambda^+_{L'}(L'/(L'\cap CK))$.
For every $(\lambda, \chi) \in \Xi'_2$ we construct a character $\widetilde \chi \in \mathfrak X(K)$ as follows.
Let $\psi \in \mathfrak X(C)$ be the restriction of $\widetilde \lambda$ to~$C$.
Since $[V_{L'}(\lambda)^*]^{(L' \cap CK)}_\chi \ne 0$, it follows that $\left.\lambda\right|_{L'\cap C} = - \left.\chi\right|_{L' \cap C}$.
As $CK = C(L' \cap CK)$ and $C \cap (L' \cap CK) = L' \cap C$, there exists a unique character $\chi_0 \in \mathfrak X(CK)$ such that $\left.\chi_0 \right|_{L' \cap CK} = \chi$ and $\left.\chi_0 \right|_C = -\psi$.
Then we put $\widetilde \chi = \left. \chi_0 \right|_K$.

By definition, we put $\Xi_2 = \lbrace (\widetilde \lambda, \widetilde \chi) \mid (\lambda, \chi) \in \Xi'_2 \rbrace$.

\begin{theorem} \label{thm_gen_2}
One has $\Xi_2 \subset \Xi \setminus \Xi_1$.
\end{theorem}

\begin{remark} \label{rem_disjoint_supports}
As can be seen from the definitions, for every $(\lambda, \chi) \in \Xi_1$ and $(\lambda',\chi') \in \Xi_2$, one has $\supp \lambda \cap \supp \lambda' = \varnothing$.
\end{remark}

We now put $\Xi_{12} = \Xi_1 \cup \Xi_2$ and $\Xi_3 = \Xi \setminus \Xi_{12}$.

To describe $\Xi_3$, we need to make special choices of $K$ and $S$ within their conjugacy classes.
As in Theorem~\ref{thm_weight_lattice}, we put $B_S = B_L \cap K$, choose $K$ and $S$ to satisfy~\textup(\ref{S1}\textup)--\textup(\ref{S6}\textup), and let $\iota \colon \mathfrak X(T) \to \mathfrak X(T \cap B_S)$ be the character restriction map.

Let $\Xi'_3$ denote the set of indecomposable elements of the monoid~$\Lambda^+_S(\mathfrak p_u / \mathfrak h_u)$.
For every $\mu \in \Xi'_3$, let $\rho_\mu$ be the element of $\Lambda_S(\mathfrak p_u / \mathfrak h_u)^\vee$ such that $\rho_\mu(\mu) = 1$ and $\rho_\mu(\mu') = 0$ for all $\mu' \in \Xi'_3 \setminus \lbrace \mu \rbrace$.
We shall also regard $\rho_\mu$ as an element of $\Lambda^\vee$ via the natural inclusion $\Lambda_S(\mathfrak p_u / \mathfrak h_u)^\vee \hookrightarrow \Lambda^\vee$ induced by~$\iota$ (see Theorem~\ref{thm_weight_lattice}).

For every $\mu \in \Xi'_3$ we let $\widetilde \mu \in \mathfrak X(T)$ be an arbitrary extension of~$\mu$ to~$T$.
For every $\alpha \in \Pi$, put $\Xi_{12}(\alpha) = \lbrace (\lambda, \chi) \in \Xi_{12} \mid \varpi_\alpha \in \supp \lambda \rbrace$.
Put also $\Pi_{12} = \lbrace \alpha \in \Pi \mid \Xi_{12}(\alpha) \ne \varnothing \rbrace$.

Now recall from the definitions in~\S\,\ref{subsec_spherical_varieties} that $\Sigma \subset \Lambda$.
Then for every $\mu \in \Xi'_3$ and $\alpha \in \Pi_{12}$ we may define
\begin{equation} \label{eqn_coeff}
\delta(\mu, \alpha) =
\begin{cases}
1 & \text{if} \ \alpha \in \Sigma, \ |\Xi_{12}(\alpha)| = 1, \ \text{and} \ \rho_\mu(\alpha) = 1; \\
0 & \text{otherwise}.
\end{cases}
\end{equation}

The following theorem expresses $\Xi_3$ via $\Xi'_3$, $\Xi_1$, $\Xi_2$, and~$\Pi \cap \Sigma$.

\begin{theorem} \label{thm_gen_3}
Under the above assumptions and notation, the following assertions hold.
\begin{enumerate}[label=\textup{(\alph*)},ref=\textup{\alph*}]
\item \label{thm_gen_3_a}
There is a bijection $\Xi'_3 \to \Xi_3$ taking each $\mu$ to an element of the form $(\lambda_\mu, \chi_\mu) = (\widetilde \mu,0) + \sum \limits_{\Omega \in \Xi_{12}} a_{\mu,\Omega} \Omega$ with $a_{\mu,\Omega} \in \ZZ$.

\item \label{thm_gen_3_b}
The coefficients $\lbrace a_{\mu,\Omega} \mid \Omega \in \Xi_{12} \rbrace$ satisfy the following system of linear equations: for every $\alpha \in \Pi_{12}$, the coefficient at $\varpi_\alpha$ in $\lambda_\mu$ equals $\delta(\mu,\alpha)$.

\item \label{thm_gen_3_c}
If $K$ is not contained in any proper parabolic subgroup of~$L$ then the coefficients $a_{\mu,\Omega}$ are uniquely determined from the above system of linear equations.
\end{enumerate}
\end{theorem}

\begin{remark}
The condition that $K$ be not contained in a proper parabolic subgroup of~$L$ can be always achieved by an appropriate choice of~$P$.
On the other hand, this condition is essential in part~(\ref{thm_gen_3_c}) of the above theorem:
Example~\ref{ex_SL3} shows that in general the uniqueness of the coefficients $a_{\mu, \Omega}$ may fail.
\end{remark}

\begin{remark}
We point out that application of Theorem~\ref{thm_gen_3} (as well as of Theorem~\ref{thm_weight_lattice}) always requires that conditions \textup(\ref{S1}\textup)--\textup(\ref{S6}\textup) be fulfilled, which may cause certain difficulties for computations in concrete examples.
Nevertheless, if $(\mathfrak l', \mathfrak l' \cap (\mathfrak c +\mathfrak k))$ is a symmetric pair (that is, $\mathfrak l' \cap (\mathfrak c +\mathfrak k)$ equals the set of fixed points of an involutive automorphism of~$\mathfrak l'$) then an embedding of $\mathfrak l' \cap (\mathfrak c +\mathfrak k)$ in $\mathfrak l'$ (and hence of $K$ in~$L$) satisfying \textup(\ref{S1}\textup)--\textup(\ref{S6}\textup) can be constructed from the corresponding Satake diagram as discussed in~\cite[\S\,26.4]{Tim}; see also the classification of all symmetric pairs along with their Satake diagrams in~\cite[Table~26.3]{Tim}.
Note that symmetric pairs form a major part in the classification of~\cite{Kr} mentioned in the Introduction.
\end{remark}

\begin{remark} \label{rem_par_ind}
If $H_u = P_u$ then the homogeneous space $G/H$ is said to be \textit{parabolically induced} from $L/K$.
In this case $\Xi_3 = \varnothing$ and hence $\Xi = \Xi_1 \cup \Xi_2$.
\end{remark}

The next two propositions provide some necessary and some sufficient conditions for an element $\alpha \in \Pi_{12}$ with $|\Xi_{12}(\alpha)|=1$ to lie in~$\Sigma$.

\begin{proposition} \label{prop_alpha_in_Sigma1}
Suppose that $\alpha \in \Pi_{12}$ and $|\Xi_{12}(\alpha)|=1$.
If $\alpha \in \Sigma$ then the following conditions hold:
\begin{enumerate}[label=\textup{(\alph*)},ref=\textup{\alph*}]
\item \label{prop_alpha_in_Sigma1_a}
$\alpha \in \Lambda$;

\item \label{prop_alpha_in_Sigma1_b}
there exists a unique $\mu \in \Xi'_3$ such that $\rho_\mu(\alpha) = 1$ and $\rho_{\mu'}(\alpha) \le 0$ for all $\mu' \in \Xi'_3 \setminus \lbrace \mu \rbrace$.
\end{enumerate}
\end{proposition}

Observe from Theorem~\ref{thm_gen_1} that $\Pi \setminus \Pi_L \subset \Pi_{12}$.
For every $\alpha \in \Pi \setminus \Pi_L$, let $I_\alpha$ denote the ideal in~$\mathfrak p_u$ generated by $\mathfrak g_{-\alpha}$.

\begin{proposition} \label{prop_alpha_in_Sigma2}
Given $\alpha \in \Pi \setminus \Pi_L$, the following assertions hold.
\begin{enumerate}[label=\textup{(\alph*)},ref=\textup{\alph*}]
\item \label{prop_alpha_in_Sigma2_a}
If $\alpha \in \Sigma$ then $I_\alpha \not\subset \mathfrak h_u$.

\item \label{prop_alpha_in_Sigma2_b}
If $I_\alpha \not\subset \mathfrak h_u$ and $\mathfrak g_{-\alpha}$ commutes with $[\mathfrak s, \mathfrak s]$ then $\alpha \in \Sigma$.
\end{enumerate}
\end{proposition}

\begin{remark}
It is shown in Example~\ref{ex_SO7} that, for $\alpha \in \Pi \setminus \Pi_L$, the conditions of Propositions~\ref{prop_alpha_in_Sigma1} and~\ref{prop_alpha_in_Sigma2}(\ref{prop_alpha_in_Sigma2_a}) do not imply $\alpha \in \Sigma$ in general.
\end{remark}

\section{Tools}
\label{sect_tools}

Throughout this section, $G$ denotes an arbitrary connected reductive group.

\subsection{Minimal parabolic subgroups and colors}

Let $H \subset G$ be a spherical subgroup.

For every $\alpha \in \Pi$, let $P_{\alpha} \supset B$ be the corresponding minimal parabolic subgroup, so that $\mathfrak p = \mathfrak g_{-\alpha} \oplus \mathfrak b$, and let $\mathcal D(\alpha) \subset \mathcal D$ be the set of $P_\alpha$-unstable colors.

The next result follows from the localization technique developed by Luna in~\cite{Lu97}; see also \cite[\S\,3.2]{Lu01}, \cite[\S\S\,30.9--30.11]{Tim}, and~\cite{Kn14}.

\begin{proposition} \label{prop_localization}
For every $\alpha \in \Pi$, the following assertions hold.
\begin{enumerate}[label=\textup{(\alph*)},ref=\textup{\alph*}]
\item \label{prop_localization_a}
$|\mathcal D(\alpha)| \le 2$ and the equality is attained if and only if $\alpha \in \Sigma$.

\item \label{prop_localization_b}
If $\alpha \in \Sigma$ then $\rho_D(\alpha) = 1$ for both $D \in \mathcal D(\alpha)$ and $\rho_D(\alpha) \le 0$ for all $D \in \mathcal D \setminus \mathcal D(\alpha)$.
\end{enumerate}
\end{proposition}

In the next two propositions, $G$ is assumed to be semisimple and simply connected.

The following result is well known; see, for instance, \cite[Lemma~2.9]{Avd_solv_inv}.

\begin{proposition} \label{prop_D_in_D(alpha)}
For every $\alpha \in \Pi$ and $D \in \mathcal D$, the condition $D \in \mathcal D(\alpha)$ holds if and only if $\varpi_\alpha \in \supp \lambda_D$.
\end{proposition}

Below is one more consequence of the localization technique, which was obtained in~\cite[\S\,2.2, Theorem~2.2]{Fos} (see also~\cite[Lemma~30.24]{Tim}).

\begin{proposition} \label{prop_Foschi}
For every $\alpha \in \Pi \cap \Sigma$ and every $D \in \mathcal D(\alpha)$, the coefficient at $\varpi_\alpha$ in~$\lambda_D$ equals~$1$.
\end{proposition}

\subsection{Reductive subgroups with characters of infinite order}

\begin{proposition} \label{prop_proper_parabolic}
Suppose that $G$ is semisimple, $H \subset G$ is a reductive \textup(but not necessarily connected\textup) subgroup, and there is a character $\chi \in \mathfrak X(H)$ of infinite order.
Then $H$ is contained in a proper parabolic subgroup of~$G$.
\end{proposition}

\begin{proof}
Let $C_0$ be the connected center of $H^0$ and put $V = \mathfrak X(C_0) \otimes_\ZZ \QQ$.
Clearly, $\chi$ restricts to a nontrivial element $\chi_0 \in \mathfrak X(C_0)$, which we identify with its image in~$V$.
Consider the action of $H$ on itself by conjugation.
As $C_0$ is preserved by this action and acted on trivially by~$H^0$, we obtain an induced action of the finite group $H/H^0$ on $\mathfrak X(C_0)$ (and hence on~$V$).
Observe that~$\chi_0$ is $H/H^0$-fixed, hence there is an $H/H^0$-stable subspace $V_0 \subset V$ complementary to $\QQ\chi_0$.
Let $Z_0 \subset C_0$ be the one-dimensional subtorus determined by the sublattice $V_0 \cap \mathfrak X(C_0)$ of~$\mathfrak X(C_0)$.
Since $\chi_0$ is $H$-fixed, it follows that $H$ acts trivially on~$Z_0$ and hence is contained in the centralizer of~$Z_0$ in~$G$.
It remains to recall that the centralizer of a nontrivial torus in a semisimple group is a reductive subgroup contained in a proper parabolic subgroup.
\end{proof}

\subsection{A general property of smooth morphisms}

Let $X,Y$ be two smooth irreducible algebraic varieties and let $\phi \colon X \to Y$ be a morphism.
Recall from \cite[Ch.~III, Proposition~10.4]{Har} that $\phi$ is \textit{smooth} if and only if for every point $x \in X$ the induced map $T_x X \to T_{\phi(x)} Y$ of the Zariski tangent spaces is surjective.
In this case, each irreducible component of each fiber of~$\phi$ has dimension $\dim X - \dim Y$ by~\cite[Ch.~III, Theorem~10.2]{Har}.

The following result is known to specialists; we provide it with a proof due to the lack of a reference.

\begin{proposition} \label{prop_pullback}
Suppose that $\phi$ is smooth and surjective.
Let $D \subset Y$ be a prime divisor and put $E=\phi^{-1}(D)$ regarded as a subset of~$X$.
Regard $D$ as a prime Cartier divisor on~$Y$ and let $\phi^*D$ denote the pullback of~$D$ by~$\phi$.
Then
\begin{enumerate}[label=\textup{(\alph*)},ref=\textup{\alph*}]
\item \label{prop_pullback_a}
$\phi^*D$ is reduced as a Cartier divisor on~$X$;

\item \label{prop_pullback_b}
if $\phi$ has irreducible fibers then $E$ is irreducible and $E = \phi^*D$ as Cartier divisors on~$X$.
\end{enumerate}
\end{proposition}

\begin{proof}
(\ref{prop_pullback_a})
To prove the assertion, we need to show that $\phi^*D$ has no multiplicities.
Choose an open subset $Z \subset Y$ such that $D \cap Z$ is nonempty and smooth; then it suffices to prove that $\phi^*(D \cap Z)$ is smooth (regarded as a closed subscheme of~$\phi^{-1}(Z)$).
Let $y \in D \cap Z$, $x \in \phi^{-1}(y)$ and let $f$ be a local equation of $D \cap Z$ at~$y$.
Then the local equation of $\phi^*D$ at~$x$ is $\phi^*f = f \circ \phi$.
Since $D \cap Z$ is smooth at $y$, the differential $df_y \colon T_yY \to \CC$ is surjective.
In view of the smoothness of~$\phi$, we conclude that the differential $d(\phi^*f)_x \colon T_xX \to \CC$ is also surjective, hence $\phi^*(D \cap Z)$ is smooth at~$x$, which completes the proof.

(\ref{prop_pullback_b})
Thanks to~(\ref{prop_pullback_a}) it remains to show that $E$ is irreducible.
As $E$ is the preimage of a prime divisor, each of its irreducible components is a prime divisor in~$X$.
The remaining argument is based on the fiber dimension theorem~\cite[Ch.~II, Exercise~3.22(e)]{Har}.
Since all the fibers of $\phi$ are irreducible and of the same dimension, we conclude that each irreducible component of $E$ maps dominantly to~$D$.
Recall that the image of a dominant morphism contains a nonempty open subset, therefore $D$ has a nonempty open subset of points $y$ such that $\phi^{-1}(y)$ meets all irreducible components of~$E$.
As $\phi^{-1}(y)$ is irreducible, the latter is possible only if $E$ has exactly one irreducible component.
\end{proof}

\subsection{Some properties of homogeneous bundles}

Let $N$ be a connected linear group and let $H \subset N$ be a subgroup.
Suppose $X$ is a smooth irreducible $N$-variety equipped with an $N$-equivariant map $\phi \colon X \to N/H$ and the fiber $Y = \phi^{-1}(eH)$ is irreducible. (Then $Y$ is automatically smooth by~\cite[Ch.~III, Corollary~10.7 and Theorem~10.2]{Har}.)
In the terminology of~\cite[\S\,4.8]{PV} or~\cite[\S\,2.1]{Tim}, $X$ is a \textit{homogeneous bundle} over~$N/H$.

The following proposition is straightforward.

\begin{proposition} \label{prop_lb_pb}
Suppose $X=N/H_1$ for a subgroup $H_1 \subset H$.
Then for every $\chi \in \mathfrak X(H)$ the pullback $\phi^*(N *_{H} \CC_\chi)$ is an $N$-linearized line bundle on $N/H_1$ isomorphic to $N *_{H_1} \CC_{\overline \chi}$ where $\overline \chi \in \mathfrak X(H_1)$ is the restriction of $\chi$ to~$H_1$.
\end{proposition}

Let $D$ be an $N$-stable prime divisor in~$X$ and let $E$ be the Cartier divisor on~$Y$ defined as the restriction of $D$ to~$Y$, so that the support of $E$ is $D \cap Y$.
Note that $E$ is $H$-stable.

\begin{proposition} \label{prop_restriction}
The following assertions hold.
\begin{enumerate}[label=\textup{(\alph*)},ref=\textup{\alph*}]
\item \label{prop_restriction_a}
$E$ is a reduced divisor \textup(that is, without multiplicities\textup).

\item \label{prop_restriction_b}
The induced action of $H$ on the set of irreducible components of~$D \cap Y$ is transitive.
In particular, $E$ is a \textup(reduced\textup) prime divisor whenever $H$ is connected.
\end{enumerate}
\end{proposition}

\begin{proof}
(\ref{prop_restriction_a})
Consider the morphism $\theta \colon N \times Y \to X$, $(g,y) \mapsto gy$.
Clearly, $\theta$ is surjective.
Thanks to~\cite[Proposition~4.22]{PV}, for every point $y \in Y$ one has $T_yX = T_y(Ny) \oplus T_yY$, therefore $\theta$ is smooth.
Now observe that $\theta^{-1}(D) = N \times (D \cap Y)$, so that $\theta^*D = N \times E$.
Proposition~\ref{prop_pullback}(\ref{prop_pullback_a}) implies that $N \times E$ is reduced, hence so is~$E$.

(\ref{prop_restriction_b})
Let $E_0$ be an irreducible component of~$D \cap Y$.
Observe that $NE_0$ is irreducible and closed in~$X$, hence $NE_0 = D$.
Then $HE_0 = NE_0 \cap Y = D \cap Y$.
\end{proof}

\section{Proofs of the main results}
\label{sect_proofs}

In this subsection, we retain the notation of~\S\,\ref{sect_main_results}.
Besides, for every group $N$ and every irreducible $N$-variety~$X$ we denote by $\mathcal D_N(X)$ the set of $N$-stable prime divisors in~$X$.

\subsection{Preparations}
\label{subsec_preparations}

According to the fixed regular embedding of $H$ in~$P$, we split $\mathcal D$ into a disjoint union $\mathcal D = \mathcal D_1 \cup \mathcal D_2 \cup \mathcal D_3$ as follows.

Firstly, we consider the natural surjective map $\phi_1 \colon G/H \to G/P$ and put
\[
\mathcal D_1 = \lbrace \phi_1^{-1}(E) \mid E \in \mathcal D_B(G/P) \rbrace.
\]
As the fiber $\phi_1^{-1}(eP) \simeq P/H$ is connected, we have $\mathcal D_1 \subset \mathcal D$ thanks to Proposition~\ref{prop_pullback}(\ref{prop_pullback_b}).
Note that $\mathcal D \setminus \mathcal D_1$ consists of all colors of $G / H$ that map dominantly to $G/P$.
Let $O_1$ be the open $B$-orbit in $G/P$; it is the complement to the union of all colors in~$G/P$.
Then $\phi_1^{-1}(O_1) = (G/H) \setminus \bigcup \limits_{D \in \mathcal D_1} D$ and the map $D \mapsto D \cap \phi_1^{-1}(O_1)$ yields a bijection $\mathcal D \setminus \mathcal D_1 \to \mathcal D_B(\phi_1^{-1}(O_1))$.
As $BP$ is open in~$G$, the base point $eP \in G/P$ belongs to~$O_1$.
In what follows we identify $\phi_1^{-1}(eP)$ with~$P/H$.
The stabilizer of $eP$ in $B$ equals $B \cap P = B_L$, therefore by Proposition~\ref{prop_restriction} the map $D \mapsto D \cap \phi_1^{-1}(eP)$ yields a bijection between the sets $D_B(\phi_1^{-1}(O_1))$ and $\mathcal E = \mathcal D_{B_L}(P/H)$.

Up to this moment, we have established the following chain of bijections:
\begin{equation} \label{eqn_chain1}
\mathcal D \setminus \mathcal D_1 \longleftrightarrow \mathcal D_B(\phi_1^{-1}(O_1)) \longleftrightarrow \mathcal E.
\end{equation}

Now consider the natural $L$-equivariant map $\phi_2 \colon P/H \to L/K$ and put
\[
\mathcal E_2 = \lbrace \phi_2^{-1}(E) \mid E \in \mathcal D_{B_L}(L/K) \rbrace.
\]
As the fiber $\phi_2^{-1}(eK) \simeq P_u/H_u$ is connected, we have
$\mathcal E_2 \subset \mathcal E$ thanks to Proposition~\ref{prop_pullback}(\ref{prop_pullback_b}).
We put $\mathcal E_3 = \mathcal E \setminus \mathcal E_2$ and for $i=2,3$ we let $\mathcal D_i$ be the subset of $\mathcal D \setminus \mathcal D_1$ corresponding to~$\mathcal E_i$ via the chain of bijections~(\ref{eqn_chain1}).
Thus we get disjoint unions
\[
\mathcal E = \mathcal E_2 \cup \mathcal E_3 \quad \text{and} \quad \mathcal D = \mathcal D_1 \cup \mathcal D_2 \cup \mathcal D_3.
\]
Note that $\mathcal E_3$ consists of all elements of $\mathcal E$ that map dominantly to~$L/K$.
Let $O_2$ be the open $B_L$-orbit in~$L/K$; it is the complement to the union of all colors in~$L/K$.
Then $\phi_2^{-1}(O_2) = (P/H)\setminus \bigcup\limits_{D \in \mathcal E_2} D$ and the map $D \mapsto D \cap \phi_2^{-1}(O_2)$ yields a bijection $\mathcal E_3 \to \mathcal D_{B_L}(\phi_2^{-1}(O_2))$.
As $B_LK$ is open in~$L$, the base point $eK \in L/K$ belongs to~$O_2$.
In what follows we identify $\phi_2^{-1}(eK)$ with~$P_u/H_u$.
The stabilizer of $eK$ in $B_L$ equals $B_L \cap K = B_S$.
Taking into account Remarks~\ref{rem_Montagard} and~\ref{rem_sm_B-stable}, we find that for every $D \in \mathcal E_3$ each irreducible component of the intersection $D \cap \phi_2^{-1}(eK)$ is $B_S$-stable.
Then it follows from Proposition~\ref{prop_restriction}(\ref{prop_restriction_b}) that $D \cap \phi_2^{-1}(eK)$ is irreducible itself and hence a prime divisor in $\phi_2^{-1}(eK)$, therefore the map $D \mapsto D \cap \phi_2^{-1}(eK)$ yields a bijection between $\mathcal E_3$ and the set $\mathcal F = \mathcal D_{B_S}(P_u/H_u)$.

We have established the following chain of bijections:

\begin{equation} \label{eqn_chain2}
\mathcal E_3 \longleftrightarrow \mathcal D_{B_L}(\phi_2^{-1}(O_2)) \longleftrightarrow \mathcal F.
\end{equation}

Now consider the natural map $\phi_3 \colon L/K \to L/CK$.
As the fiber
\[
\phi_3^{-1}(eCK) \simeq CK/K \simeq C/(C \cap K)
\]
is connected, again by Proposition~\ref{prop_pullback}(\ref{prop_pullback_b}) we obtain
\[
\lbrace \phi_3^{-1}(E) \mid E \in \mathcal D_{B_L}(L/CK) \rbrace \subset \mathcal D_{B_L}(L/K).
\]
Let $O_3$ denote the open $B_L$-orbit in~$L/CK$.
As $B_LCK$ is open in~$L$, the base point $eCK \in L/CK$ belongs to~$O_3$.
Clearly, the stabilizer of $eCK$ in $B_L$ contains~$C$, hence it acts transitively on~$\phi_3^{-1}(eCK)$.
The latter means that the map $E \mapsto \phi_3^{-1}(E)$ is a bijection between $\mathcal D_{B_L}(L/CK)$ and $\mathcal D_{B_L}(L/K)$.
As was observed in Remark~\ref{rem_L/K}, $L'$ acts transitively on $L/CK$, therefore we have an isomorphism $L/CK \simeq L'/(L' \cap CK)$ as $L'$-varieties.
Finally, since $C$ acts trivially on $L/CK$, each color of $L/CK$ is in fact a color of $L'/(L' \cap CK)$ with respect to the Borel subgroup $B_{L'} = B_L \cap L' \subset L'$.

As a result of the above considerations, we get a chain of bijections
\begin{equation} \label{eqn_chain3}
\mathcal E_2 \longleftrightarrow \mathcal D_{B_L}(L/K) \longleftrightarrow \mathcal D_{B_L}(L/CK) \longleftrightarrow \mathcal D_{B_{L'}}(L'/(L'\cap CK)).
\end{equation}

\subsection{Proofs}

At first, we mention the following straightforward result.

\begin{proposition} \label{prop_ewm_for_G/P}
The monoid $\widehat \Lambda^+_G(G/P)$ is freely generated by the elements $(\varpi_\alpha, - \varpi_\alpha)$ with $\alpha \in \Pi \setminus \Pi_L$.
\end{proposition}

\begin{proof}[{Proof of Theorem~\textup{\ref{thm_rank}}}]
According to Theorem~\ref{thm_ewm_is_free}, we have
\[
\rk \widehat \Lambda^+ = |\mathcal D| = |\mathcal D_1| + |\mathcal D_2| + |\mathcal D_3|.
\]
It follows from the considerations in~\S\,\ref{subsec_preparations}, Theorem~\ref{thm_ewm_is_free}, and Proposition~\ref{prop_ewm_for_G/P} that $|\mathcal D_1| = |\mathcal D_B(G/P)| = \rk \widehat \Lambda^+_G(G/P) = |\Pi \setminus \Pi_L|$.
Chain~(\ref{eqn_chain3}) along with Theorem~\ref{thm_ewm_is_free} imply $|\mathcal D_2| = |\mathcal E_2| = \rk \widehat \Lambda^+_{L'}(L'/(L'\cap CK))$.
Chain~(\ref{eqn_chain2}) implies
\[
|\mathcal D_3| = |\mathcal E_3| = |\mathcal F| =  \rk \Lambda^+_S(\mathfrak p_u / \mathfrak h_u)
\]
thanks to Remarks~\ref{rem_Montagard} and~\ref{rem_sm_B-stable}.
\end{proof}

Theorem~\ref{thm_gen_1} is implied by

\begin{proposition} \label{prop_Xi1_D1}
$\Xi_1 = \lbrace (\lambda_D,\chi_D) \mid D \in \mathcal D_1 \rbrace$.
\end{proposition}

\begin{proof}
Take any $D \in \mathcal D_1$ and let $E \in \mathcal D_B(G/P)$ be such that $D = \phi_1^{-1}(E)$.
Consider the homogeneous line bundle $\mathcal L = G *_P \CC_{-\chi_E}$ and a section $s_E \in \Gamma(\mathcal L)$ such that $\div s_E = E$.
By Proposition~\ref{prop_lb_pb}, there is an isomorphism $\varphi_1^*\mathcal L \simeq G*_H \CC_{-\overline \chi_E}$ of $G$-linearized line bundles on~$G/H$.
Clearly, the $B$-weight of $\phi_1^*s_E$ is~$\lambda_D$.
As $D = \div (\phi_1^*s_E)$ by Proposition~\ref{prop_pullback}(\ref{prop_pullback_b}), we finally obtain $(\lambda_D,\chi_D) = (\lambda_E, \overline \chi_E)$.
It remains to make use of Proposition~\ref{prop_ewm_for_G/P}.
\end{proof}

Theorem~\ref{thm_gen_2} is implied by

\begin{proposition} \label{prop_Xi2_D2}
$\Xi_2 = \lbrace (\lambda_D,\chi_D) \mid D \in \mathcal D_2 \rbrace$.
\end{proposition}

\begin{proof}
Take any $D \in \mathcal D_2$ and let $D' \in \mathcal D_B(\phi_1^{-1}(O_1))$, $E \in \mathcal E_2$, $E' \in \mathcal D_{B_L}(L/K)$, $E'' \in \mathcal D_{B_L}(L/CK)$ be the objects corresponding to $D$ via the chains of bijections~(\ref{eqn_chain1}) and~(\ref{eqn_chain3}).
The dominant weights of $L'$ will be considered with respect to the Borel subgroup $B_{L'} = L' \cap B_L$.

\textit{Step}~\newstep.
Put $X =L/CK$ and first regard $X$ as a homogeneous space for the action of $L'$, so that $X \simeq L'/(L'\cap CK)$ as $L'$-varieties.
Then $E''$ may be regarded as a color of~$L'/(L'\cap CK)$, and we let $(\lambda,\chi) = (\lambda_{E''},\chi_{E''}) \in \Lambda^+_{L'} \times \mathfrak X(L' \cap CK)$ be the corresponding biweight.
Let $\mathcal L$ be the $L'$-linearized line bundle on~$X$ together with a section $s \in \Gamma(\mathcal L)$ such that $E'' = \div s$.
Then $s$ is $B_{L'}$-semiinvariant of weight~$\lambda$ and $L' \cap CK$ acts on the fiber of $\mathcal L$ over~$eCK$ via the character~$-\chi$.

\textit{Step}~\newstep.
The finite group $L' \cap C$ acts trivially on~$X$, hence it acts on each fiber of $\mathcal L$ via the same character, which equals the restriction to~$L' \cap C$ of both~$-\lambda$ and~$-\chi$.
We now recall the notation preceding the statement of Theorem~\ref{thm_gen_2}.
Recall the map $\tau_L$ given by~(\ref{eqn_restr}) and put $\widetilde \lambda = \tau_L^{-1}(\lambda)$.
Let $\psi \in \mathfrak X(C)$ be the restriction of $\widetilde \lambda$ to~$C$ and equip $\mathcal L$ with an $L$-linearization by letting $C$ act on each fiber of~$\mathcal L$ via the character~$-\psi$. (Recall that $C$ acts trivially on~$X$.)
Then $CK$ acts on the fiber over $eCK$ via the character $-\chi_0 \in \mathfrak X(CK)$ whose restrictions to $L' \cap CK$ and~$C$ equal $-\chi$ and~$-\psi$, respectively.
Thus we have identified $\mathcal L$, regarded as an $L$-linearized line bundle on~$X$, with $L *_{CK} \CC_{-\chi_0}$ and $s$ with an element of $\Gamma(L *_{CK} \CC_{-\chi_0})^{(B_L)}_{\widetilde \lambda}$.

\textit{Step}~\newstep. \label{D2_step3}
Let $\widetilde \chi$ be the restriction of $\chi_0$ to~$K$.
Proposition~\ref{prop_lb_pb} yields an isomorphism $\phi_3^*\mathcal L \simeq L *_K \CC_{-\widetilde \chi}$ and the $B_L$-weight of $\phi_3^*s$ is~$\widetilde \lambda$.
Thanks to Proposition~\ref{prop_pullback}(\ref{prop_pullback_b}), we have $E' = \div(\phi_3^*s)$, which in particular implies $(\widetilde \lambda, \widetilde \chi) \in \widehat \Lambda^+_L(L/K)$.

\textit{Step}~\newstep.
Equip $L/K$ and $\phi_3^*\mathcal L$ with the trivial action of $P_u$; then $L/K \simeq P/(K P_u)$ as a homogeneous space for~$P$ and $\phi_3^*\mathcal L \simeq P *_{K P_u} \CC_{-\widetilde \chi}$ as a $P$-linearized line bundle on~$L/K$.
Arguments as in Step~\ref{D2_step3} yield an isomorphism $\phi_2^*\phi_3^*\mathcal L \simeq P *_H \CC_{-\widetilde \chi}$ of $P$-linearized line bundles on~$P/H$ and show that the $B_L$-weight of $\phi_2^*\phi_3^*s$ is~$\widetilde \lambda$ and $E = \div \phi_2^*\phi_3^*s$.
In what follows we identify $\phi_2^*\phi_3^*s$ with an element $s_E \in \Gamma(P *_H \CC_{-\widetilde \chi})^{(B_L)}_{\widetilde \lambda}$.

\textit{Step}~\newstep.
It is easy to see that the $P$-linearized line bundle $P *_H \CC_{-\widetilde \chi}$ on $P/H$ can be regarded as a subbundle of the $G$-linearized line bundle $G *_H \CC_{-\widetilde \chi}$.
Furthermore, $s_E$~extends to a $B$-semiinvariant rational section $s_D$ of $G *_H \CC_{-\widetilde \chi}$ of weight $\widetilde \lambda$ that is regular over the open subset $\phi_1^{-1}(O_1) \subset G/H$.
Thanks to Proposition~\ref{prop_restriction}, the divisor of zeros of $s_D$ over $\phi_1^{-1}(O_1)$ is $D'$, which implies $\div s_D = D + \sum \limits_{F \in \mathcal D_1}c_F F$ with $c_F \in \ZZ$.

\textit{Step}~\newstep.
Let $W \subset V_G(\widetilde \lambda)^*$ be the subspace of $P_u$-invariant vectors.
Then $W$ is an $L$-module.
Clearly, $W$ contains the line spanned by a lowest-weight vector of $V_G(\widetilde \lambda)^*$, therefore $W$ contains an $L$-submodule dual to $V_L(\widetilde \lambda)$.
We have seen in Step~\ref{D2_step3} that $(\widetilde \lambda, \widetilde \chi) \in \widehat \Lambda^+_L(L/K)$, therefore $W^{(K)}_{\widetilde \chi} \ne 0$.
Since $H_u \subset P_u$, we have $W^{(K)}_{\widetilde \chi} \subset [V_G(\widetilde \lambda)^*]^{(H)}_{\widetilde \chi}$.
It follows that $[V_G(\widetilde \lambda)^*]^{(H)}_{\widetilde \chi} \ne 0$ and hence $(\widetilde \lambda, \widetilde \chi) \in \widehat \Lambda^+$.
Thus there exists a nonzero section $\widetilde s \in \Gamma(G *_H \CC_{-\widetilde \chi})^{(B)}_{\widetilde \lambda}$.

\textit{Step}~\newstep.
Since $s_D$ and $\widetilde s$ have the same $B$-weight, the quotient $s_D / \widetilde s$ is a $B$-invariant rational function on~$G/H$, which is necessarily constant because $G/H$ has an open $B$-orbit.
Consequently, $s_D$ is proportional to~$\widetilde s$ and hence a regular section of $G *_H \CC_{-\widetilde \chi}$, which implies $c_F \ge 0$ for all $F \in \mathcal D_1$.

\textit{Step}~\newstep.
Assume there is $F \in \mathcal D_1$ with $c_F > 0$.
Then $\div s_D - F$ is an effective $B$-stable divisor on $G/H$, therefore $(\widetilde \lambda - \lambda_F, \widetilde \chi - \chi_F) \in \widehat \Lambda^+$ by Theorem~\ref{thm_ewm_is_free}.
By construction we get $\supp \widetilde \lambda \cap \supp \lambda_F = \varnothing$, which implies $\widetilde \lambda - \lambda_F \notin \Lambda^+_G$, a contradiction.
It follows that $c_F = 0$ for all $F \in \mathcal D_1$, hence $\div s_D = D$ and $(\lambda_D, \chi_D) = (\widetilde \lambda, \widetilde \chi)$, which completes the proof.
\end{proof}

Propositions~\ref{prop_Xi1_D1} and~\ref{prop_Xi2_D2} yield

\begin{corollary} \label{crl_Xi3_D3}
$\Xi_3 = \lbrace (\lambda_D,\chi_D) \mid D \in \mathcal D_3 \rbrace$.
\end{corollary}

\begin{proof}[{Proof of Theorem~\textup{\ref{thm_gen_3}}}]
(\ref{thm_gen_3_a})
Fix $\mu \in \Xi'_3$ and let $F_\mu \in \mathcal D_{B_S}(P_u/H_u)$ be the corresponding element (see Remarks~\ref{rem_Montagard} and~\ref{rem_sm_B-stable}).
Let $E'_\mu \in \mathcal D_{B_L}(\phi_2^{-1}(O_2))$, $E_\mu \in \mathcal E_3$, $D'_\mu \in \mathcal D_B(\phi_1^{-1}(O_1))$, $D_\mu \in \mathcal D_3$ be the objects corresponding to $F_\mu$ via the chains of bijections~(\ref{eqn_chain2}) and~(\ref{eqn_chain1}).
In view of Corollary~\ref{crl_Xi3_D3} it suffices to show that the required bijection $\Xi'_3 \to \Xi_3$ is provided by the map $\mu \mapsto (\lambda_{D_\mu},\chi_{D_\mu})$.

Fix a character $\widetilde \mu \in \mathfrak X(T)$ whose restriction to~$T \cap B_S$ equals~$\mu$.

Let $f''\in \CC[P_u/H_u]^{(B_S)}_\mu$ be such that $\div f'' = F_\mu$.
We extend $f''$ to a function $f' \in \CC[\phi_2^{-1}(O_2)]^{(B_L)}_{\widetilde \mu}$ by setting $f'(b^{-1}x) = \mu(b)f''(x)$ for all $b \in B_L$ and $x \in P_u/H_u$.
This extension is well defined because $B_L \cap K = B_S$.
Then $\div f' = E'_\mu$ by Proposition~\ref{prop_restriction}(\ref{prop_restriction_a}).
Next consider the open subset $X = B(\phi_2^{-1}(O_2)) \subset G/H$ and extend $f'$ to a function $f \in \CC[X]^{(B)}_{\widetilde \mu}$ by setting $f(b^{-1}x) = \mu(b)f'(x)$ for all $b \in B$ and $x \in \phi_2^{-1}(O_2)$.
Again, this extension is well defined because $B \cap P = B_L$.
Then $\div f = D_\mu \cap X$ thanks to Proposition~\ref{prop_restriction}.
In what follows, we shall regard $f$ as a rational function on $G/H$.
Observe that $(G/H) \setminus X$ is the union of all colors in $\mathcal D_{12} = \mathcal D_1 \cup \mathcal D_2$, therefore $\div f = D_\mu - \sum \limits_{D \in \mathcal D_{12}}a_D D$ for some $a_D \in \ZZ$.
Now for every $D \in \mathcal D_{12}$ let $\mathcal L_D$ be the $G$-linearized line bundle on $G/H$ together with a section $s_D \in \Gamma(\mathcal L_D)$ such that $D = \div s_D$.
Then the section $s = f\prod\limits_{D \in \mathcal D_{12}} s_D^{a_D}$ of the line bundle $\bigotimes\limits_{D \in \mathcal D_{12}} \mathcal L_D^{\otimes a_D}$ satisfies $\div s = D_\mu$, which yields $(\lambda_{D_\mu}, \chi_{D_\mu}) = (\widetilde \mu, 0 ) + \sum \limits_{D \in \mathcal D_{12}}a_D(\lambda_D,\chi_D)$.
Applying Propositions~\ref{prop_Xi1_D1} and~\ref{prop_Xi2_D2}, we get the claim.

(\ref{thm_gen_3_b})
For every $\alpha \in \Pi$, let $c(\mu,\alpha)$ be the coefficient at $\varpi_\alpha$ in~$\lambda_{D_\mu}$.
If $\alpha \notin \Pi_{12}$ then $c(\mu,\alpha)$ is uniquely determined from~$\widetilde \mu$.
In what follows we assume $\alpha \in \Pi_{12}$.
Then Propositions~\ref{prop_localization}(\ref{prop_localization_a}), \ref{prop_D_in_D(alpha)}, and~\ref{prop_Foschi} imply that
\[
c(\mu,\alpha) =
\begin{cases}
1 & \text{if} \ \alpha \in \Sigma, \ |\mathcal D_{12}(\alpha)| = 1, \ \text{and} \ D_\mu \in \mathcal D(\alpha); \\
0 & \text{otherwise}.
\end{cases}
\]
According to Proposition~\ref{prop_localization}(\ref{prop_localization_b}), in the case $\alpha \in \Sigma$ the condition $D_\mu \in \mathcal D(\alpha)$ is equivalent to $\rho_{D_\mu}(\alpha)=1$.
As a byproduct of the arguments in~(\ref{thm_gen_3_a}) it follows that $\rho_{D_\mu}$ coincides with the extension of $\rho_\mu$ to $\Lambda$ via the inclusion $\Lambda_S(\mathfrak p_u / \mathfrak h_u)^\vee \hookrightarrow \Lambda^\vee$, which enables us to conclude that $c(\mu,\alpha) = \delta(\mu, \alpha)$.

(\ref{thm_gen_3_c})
Assume that the system of linear equations $\lbrace c(\mu,\alpha) = \delta(\mu,\alpha) \mid \alpha \in \Pi_{12} \rbrace$ in the set of indeterminates $\lbrace a_D \mid D \in \mathcal D_{12} \rbrace$ has at least two different solutions.
Then there is a nontrivial (with at least one nonzero element) collection of integers $\lbrace b_D \mid D \in \mathcal D_{12} \rbrace$ such that
\begin{equation} \label{eqn_equation}
\sum \limits_{D \in \mathcal D_{12}} b_D(\lambda_D, \chi_D) = (0,\chi)
\end{equation}
for some $\chi \in \mathfrak X(K)$.
Remark~\ref{rem_disjoint_supports}, Proposition~\ref{prop_Xi1_D1}, and the definition of~$\Xi_1$ immediately imply $b_D = 0$ for all $D \in \mathcal D_1$.
Next, given $D \in \mathcal D_2$, let $E \in \mathcal D_{B_{L'}}(L'/(L' \cap CK))$ correspond to $D$ via the chains of bijections~(\ref{eqn_chain1}) and~(\ref{eqn_chain3}) and let $(\overline \lambda_D, \overline \chi_D)$ be the indecomposable element of $\widehat \Lambda^+_{L'}(L'/(L' \cap CK))$ corresponding to~$E$.
Recall from the last conclusion of the proof of Proposition~\ref{prop_Xi2_D2} that $\overline \lambda_D = \tau_L(\lambda_D)$; then equation~(\ref{eqn_equation}) induces an equation of the form
\begin{equation} \label{eqn_equation2}
\sum \limits_{D \in \mathcal D_2} b_D(\overline \lambda_D, \overline \chi_D) = (0, \overline \chi)
\end{equation}
for some $\overline \chi \in \mathfrak X(L' \cap CK)$.
As the collection $\lbrace b_D \mid D \in \mathcal D_2 \rbrace$ is nontrivial and the indecomposable elements of $\widehat \Lambda^+_{L'}(L'/(L' \cap CK))$ are linearly independent, (\ref{eqn_equation2}) is possible only if $\overline \chi$ is of infinite order in $\mathfrak X(L' \cap CK)$.
Applying Proposition~\ref{prop_proper_parabolic} we find that $L' \cap CK$ is contained in a proper parabolic subgroup of~$L'$ and hence $K$ is contained in a proper parabolic subgroup of~$L$, a contradiction.
\end{proof}

\begin{proof}[Proof of Proposition~\textup{\ref{prop_alpha_in_Sigma1}}]
Part~(\ref{prop_alpha_in_Sigma1_a}) follows from the inclusion~$\Sigma \subset \Lambda$; part~(\ref{prop_alpha_in_Sigma1_b}) is implied by Propositions~\ref{prop_localization} and~\ref{prop_D_in_D(alpha)}.
\end{proof}

To prove Proposition~\ref{prop_alpha_in_Sigma2}, we shall need several auxiliary results.

Let $\alpha \in \Pi$ and let $U_{-\alpha}$ denote the one-dimensional root unipotent subgroup in~$B^-$ with Lie algebra~$\mathfrak g_{-\alpha}$.
Since $U_{-\alpha}B$ is dense in~$P_\alpha$, it follows that a color $D \in \mathcal D$ is $P_\alpha$-unstable if and only if $D$ is $U_{-\alpha}$-unstable.
Note that $U_{-\alpha}$ naturally acts on $P/H$ and, if $\alpha \in \Pi \setminus \Pi_L$, on~$P_u/H_u$.

\begin{lemma} \label{lemma_unstable}
Suppose $D \in \mathcal D \setminus \mathcal D_1$ and let $E \in \mathcal E$ correspond to~$D$ via the chain~\textup{(\ref{eqn_chain1})}.
Then the following conditions are equivalent.
\begin{enumerate}[label=\textup{(\arabic*)},ref=\textup{\arabic*}]
\item \label{lemma_unstable_1}
$D$ is $U_{-\alpha}$-unstable.

\item \label{lemma_unstable_2}
$E$ is $U_{-\alpha}$-unstable.
\end{enumerate}
Moreover, if $\alpha \in \Pi \setminus \Pi_L$, $D \in \mathcal D_3$, and $F \in \mathcal F$ corresponds to~$E$ via the chain~\textup{(\ref{eqn_chain2})} then the above conditions are equivalent to the following one:
\begin{enumerate}[label=\textup{(\arabic*)},ref=\textup{\arabic*}]
\setcounter{enumi}{2}
\item \label{lemma_unstable_3}
$F$ is $U_{-\alpha}$-unstable.
\end{enumerate}
\end{lemma}

\begin{proof}
Let $\widetilde O_1 \supset O_1$ be the open $P_\alpha$-orbit in $G/P$ and let $D'$ be the restriction of $D$ to $\phi_1^{-1}(\widetilde O_1)$.
Then $D' \cap \phi_1^{-1}(eP) = E$ and we find that $D'$ (and hence~$D$) is $P_\alpha$-stable if and only if $E$ is ($P \cap P_\alpha$)-stable.
Since $U_{-\alpha} B_L$ is open in $P \cap P_\alpha$, we get the equivalence of (\ref{lemma_unstable_1}) and~(\ref{lemma_unstable_2}).
Now suppose $\alpha \in \Pi \setminus \Pi_L$ and $D \in \mathcal D_3$.
Then $U_{-\alpha} \subset P_u$ and $P \cap P_\alpha = U_{-\alpha}B_L$.
Let $\widetilde O_2 \supset O_2$ be the open $U_{-\alpha}B_L$-orbit in~$L/K \simeq P/(K P_u)$ and let $E'$ be the restriction of $E$ to $\phi_2^{-1}(\widetilde O_2)$.
Then $E' \cap \phi_2^{-1}(eK) = F$ and we find that $E'$ (and hence~$E$) is $U_{-\alpha}B_L$-stable if and only if $F$ is $U_{-\alpha}B_S$-stable, whence the equivalence of~(\ref{lemma_unstable_2}) and~(\ref{lemma_unstable_3}).
\end{proof}

In what follows we suppose that $\alpha \in \Pi \setminus \Pi_L$.
Recall from \S\,\ref{sect_main_results} that $I_\alpha$ is the ideal in $\mathfrak p_u$ generated by $\mathfrak g_{-\alpha}$.

\begin{lemma} \label{lemma_acts_nontrivially}
The following conditions are equivalent:
\begin{enumerate}[label=\textup{(\arabic*)},ref=\textup{\arabic*}]
\item
$I_\alpha \not\subset \mathfrak h_u$.

\item
$U_{-\alpha}$ acts nontrivially on~$P_u/H_u$.
\end{enumerate}
\end{lemma}

\begin{proof}
Let $N_{\alpha}$ be the normal subgroup of $P_u$ generated by $U_{-\alpha}$.
Then $I_\alpha$ is precisely the Lie algebra of $N_{\alpha}$.
If $U_{-\alpha}$ acts trivially on~$P_u/H_u$ then so does~$N_\alpha$, therefore $N_\alpha \subset H_u$ and hence $I_\alpha \subset \mathfrak h_u$, a contradiction.
Conversely, if $I_\alpha \subset \mathfrak h_u$ then $N_\alpha \subset H_u$.
Since $N_\alpha$ is a normal subgroup of~$P_u$, it acts trivially on $P_u/H_u$, hence so does~$U_{-\alpha}$.
\end{proof}

\begin{lemma} \label{lemma_alpha_in_Sigma}
The following conditions are equivalent:
\begin{enumerate}[label=\textup{(\arabic*)},ref=\textup{\arabic*}]
\item
$\alpha \in \Sigma$.

\item
$\mathcal F$ contains a $U_{-\alpha}$-unstable element \textup(which is automatically unique\textup).
\end{enumerate}
\end{lemma}

\begin{proof}
It follows from Propositions~\ref{prop_D_in_D(alpha)}, \ref{prop_Xi1_D1}, \ref{prop_Xi2_D2} and Remark~\ref{rem_disjoint_supports} that $U_{-\alpha}$ moves exactly one color in $\mathcal D_1 \cup \mathcal D_2$.
Making use of Proposition~\ref{prop_localization}(\ref{prop_localization_a}), we find that $\alpha \in \Sigma$ if and only if $U_{-\alpha}$ moves exactly one color in~$\mathcal D_3$.
It remains to apply Lemma~\ref{lemma_unstable}.
\end{proof}

\begin{proof}[Proof of Proposition~\textup{\ref{prop_alpha_in_Sigma2}}]
Part~(\ref{prop_alpha_in_Sigma2_a}) follows by combining Lemmas~\ref{lemma_alpha_in_Sigma} and~\ref{lemma_acts_nontrivially}.

(\ref{prop_alpha_in_Sigma2_b})
By Lemma~\ref{lemma_acts_nontrivially}, the first condition says that $U_{-\alpha}$ acts nontrivially on~$P_u/H_u$, hence on $\CC[P_u/H_u]$.
The second condition means that $U_{-\alpha}$ commutes with~$S'$ and hence is normalized by~$S$.
Recall from Remark~\ref{rem_Montagard} that $P_u/H_u \simeq \mathfrak p_u/\mathfrak h_u$ as $S$-varieties.
For every $\mu \in \Lambda^+_S(P_u/H_u)$ fix a nonzero function $f_\mu \in \CC[P_u/H_u]^{(B_S)}_\mu$.
Since every simple $S$-submodule in $\CC[P_u/H_u]$ equals the linear span of the $S'$-orbit of a highest-weight vector, it follows that there exists $\nu \in \Lambda^+_S(P_u/H_u)$ such that $f_\nu$ is $U_{-\alpha}$-unstable.
As $f_\nu$ is proportional to the product of several functions in $\lbrace f_\mu \mid \mu \in \Xi'_3\rbrace$, we find that $f_\mu$ is $U_{-\alpha}$-unstable for some~$\mu \in \Xi'_3$.
Then $F = \div f_\mu$ lies in~$\mathcal F$ by Remark~\ref{rem_sm_B-stable}, and $F$ is $U_{-\alpha}$-unstable since $U_{-\alpha}$ has no nontrivial characters.
It remains to apply Lemma~\ref{lemma_alpha_in_Sigma}.
\end{proof}

\section{The strongly solvable case}
\label{sect_strongly_solvable}

In this subsection we obtain a new proof of the main result of~\cite{AvdG} by deducing it from the main results of this paper.

We retain the notation of \S\,\ref{sect_main_results} and assume $P = B^-$, so that $H$ is strongly solvable.
Then $K = S$ and so $\mathfrak X(H)$ is identified with $\mathfrak X(S)$.

According to Theorem~\ref{thm_criterion_spherical}, $H$ is spherical if and only if $\mathfrak p_u / \mathfrak h_u$ is a spherical $S$-module; in what follows we assume that these conditions are fulfilled.
Let $\iota \colon \mathfrak X(T) \to \mathfrak X(S)$ be the character restriction map.

A root $\alpha \in \Delta^+$ is said to be \textit{active} if $\mathfrak g_{-\alpha} \not \subset \mathfrak h$.
Let $\Psi \subset \Delta^+$ be the set of all active roots.
Let $\Phi \subset \mathfrak X(S)$ be the set of weights for the action of $S$ on $(\mathfrak p_u / \mathfrak h_u)^*$.
Then $\Phi = \iota(\Psi) = \Xi'_3$.

According to~\cite[Proposition~3]{Avd_solv}, for every $\alpha \in \Psi$ there exists a unique simple root $\pi(\alpha) \in \Supp \alpha$ with the following property:
if $\alpha = \beta + \gamma$ for some $\beta, \gamma \in \Delta^+$ then $\beta \in \Psi$ \textup(resp. $\gamma \in \Psi$\textup) if and only if $\pi(\alpha) \notin \Supp \beta$ \textup(resp. $\pi(\alpha) \notin \Supp \gamma$\textup).
The above property yields a well-defined map $\pi \colon \Psi \to \Pi$.

For every $\varphi \in \Phi$, we put $\lambda_\varphi = \sum \limits_{\alpha \in \pi(\Psi \cap \iota^{-1}(\varphi))}\varpi_\alpha$.

\begin{theorem}[{\cite[Theorem~4]{AvdG}; see also \cite[Proposition~5.22]{Avd_solv_inv}}] \label{thm_AG}
The monoid $\widehat \Lambda^+$ is freely generated by all the elements $(\varpi_\alpha, -\iota(\varpi_\alpha))$ with $\alpha \in \Pi$ and all the elements $(\lambda_\varphi, - \iota(\lambda_\varphi) + \varphi)$ with $\varphi \in \Phi$.
\end{theorem}

For every $\beta \in \Psi$, we put $F(\beta) = \lbrace \beta \rbrace \cup \lbrace \gamma \in \Psi \mid \beta - \gamma \in \Delta^+\rbrace$.
In what follows, we shall need the following result obtained in~\cite[Corollary~6]{Avd_solv}.

\begin{proposition} \label{prop_pi_bijective}
For every $\beta \in \Psi$, the restricted map $\pi \colon F(\beta) \to \Supp \beta$ is bijective.
\end{proposition}

\begin{proof}[{Proof of Theorem~\textup{\ref{thm_AG}}}]
We apply our theorems from~\S\,\ref{sect_main_results}.
Clearly,
\[
\Xi_1 = \lbrace (\varpi_\alpha, -\iota(\varpi_\alpha)) \mid \alpha \in \Pi \rbrace \quad \text{and} \quad \Xi_2 = \varnothing.
\]
Next, for every $\varphi \in \Phi$ we fix an active root $\mu_\varphi \in \Psi$ such that $\iota(\mu_\varphi) = \varphi$.
Then Theorem~\ref{thm_gen_3}(\ref{thm_gen_3_a}) asserts the existence of a bijection $\Xi'_3 \to \Xi_3$, $\varphi \mapsto \Omega_\varphi$, such that $\Omega_\varphi \in (\mu_\varphi, 0) + \ZZ\Xi_1$.
As every element of $\mathfrak X(T)$ is an integer linear combination of the fundamental weights, for every $\lambda \in \mathfrak X(T)$ we have $(\lambda, -\iota(\lambda)) \in \ZZ \Xi_1$.
Consequently, $(\mu_\varphi, -\varphi) \in \ZZ \Xi_1$ and we can express $\Omega_\varphi$ in the form
\[
\Omega_\varphi = (0,\varphi) + \sum \limits_{\alpha \in \Pi} c_{\varphi,\alpha}(\varpi_{\alpha}, -\iota(\varpi_\alpha))
\]
with $c_{\varphi,\alpha} \in \ZZ$.
Then Theorem~\ref{thm_gen_3}(\ref{thm_gen_3_b}) yields $c_{\varphi,\alpha} = \delta(\varphi,\alpha)$ for all $\varphi \in \Phi$ and $\alpha \in \Pi$.

Now fix $\alpha \in \Pi$.
By~\cite[Lemma~4.4]{GaPe}, the ideal $I_\alpha$ equals the sum of all~$\mathfrak g_{-\beta}$ with $\beta \in \Delta^+$ and $\alpha \in \Supp \beta$.
If $\alpha \notin \bigcup \limits_{\beta \in \Psi} \Supp \beta$ then $I_\alpha\subset \mathfrak h_u$ and hence $\alpha \notin \Sigma$ by Proposition~\ref{prop_alpha_in_Sigma2}(\ref{prop_alpha_in_Sigma2_a}), which implies $c_{\varphi,\alpha} = 0$ for all $\varphi \in \Phi$.
In what follows we assume that $\alpha \in \Supp \beta$ for some $\beta \in \Psi$.
Clearly, $I_\alpha\not\subset \mathfrak h_u$ and $[\mathfrak s, \mathfrak s] = 0$, therefore $\alpha \in \Sigma$ by Proposition~\ref{prop_alpha_in_Sigma2}(\ref{prop_alpha_in_Sigma2_b}).

Making use of Proposition~\ref{prop_pi_bijective} we may assume in addition that $\alpha = \pi(\beta)$.
Put $\varphi_\alpha = \iota(\beta)$; then \cite[Lemma~10]{Avd_solv} guarantees that $\varphi_\alpha$ does not depend on the choice of $\beta \in \Psi$ with $\pi(\beta) = \alpha$.

By~\cite[Corollary~5.12(a)]{Avd_solv_inv}, one has $\alpha = \sum \limits_{\gamma \in F(\beta)} b_\gamma \gamma$ with $b_\gamma \in \ZZ$.
Using the defining property of~$\pi$, we find that $b_\beta = 1$.
Since the set $\lbrace \iota(\gamma) \mid \gamma \in F(\beta) \rbrace$ is linearly independent (see~\cite[Lemma~5]{Avd_solv}), it follows that $\rho_{\varphi_\alpha}(\alpha) = 1$.
Applying Theorem~\ref{thm_gen_3}(\ref{thm_gen_3_b}) and Proposition~\ref{prop_alpha_in_Sigma1}(\ref{prop_alpha_in_Sigma1_b}) we find that $c_{\varphi_\alpha, \alpha}=1$ and $c_{\varphi,\alpha} = 0$ for all $\varphi \in \Phi \setminus \lbrace \varphi_\alpha \rbrace$.

As a summary of the above considerations, for every~$\varphi \in \Phi$ and $\alpha \in \Pi$ we obtain
\[
c_{\varphi, \alpha} =
\begin{cases}
1 & \text{if} \ \alpha \in \pi(\Psi \cap \iota^{-1}(\varphi));\\
0 & \text{otherwise},
\end{cases}
\]
which completes the proof.
\end{proof}

\section{Examples}
\label{sect_examples}

In this section, we present three examples of computing the extended weight monoids, which demonstrate several features.
For simplicity, we write $\varpi_i$ instead of $\varpi_{\alpha_i}$.
In all examples, $G$ is a classical group and we choose $B,B^-,T$ to be the subgroup of all upper triangular, lower triangular, diagonal matrices, respectively, contained in~$G$.
The notation $\varpi'_i$ stands for the restriction of $\varpi_i$ to~$T \cap L'$ (which is a fundamental weight of~$L'$).

\begin{example} \label{ex_SL6}
Let $G = \SL_6$, then $\Pi = \lbrace \alpha_1,\ldots, \alpha_5 \rbrace$ where $\alpha_i(t) = t_it_{i+1}^{-1}$ for all $i = 1,\ldots,5$ and $t = \diag(t_1,\ldots,t_6) \in T$.
Let $P \subset G$ be the parabolic subgroup consisting of all matrices in~$G$ of the form
\begin{equation} \label{eqn_example1}
\begin{pmatrix} A & 0 \\ Y & B \end{pmatrix} \quad \text{(all blocks are of size~$3\times 3$)}
\end{equation}
and let $H$ be the subgroup of~$P$ defined by $B=(A^T)^{-1}$ and $Y = Y^T$.
Then $L$ (resp.~$K$) is defined in~$P$ (resp.~$H$) by $Y = 0$,
\[
C = \lbrace \diag(c,c,c,c^{-1},c^{-1},c^{-1}) \mid c \in \CC^\times \rbrace,
\]
and $\mathfrak p_u$ (resp.~$\mathfrak h_u$) consists of all matrices of the form~(\ref{eqn_example1}) satisfying $A=B=0$ (resp.~$A=B=0$ and $Y^T=Y$).
Note that $CK=K$.

Let $\chi$ be the character of $H$ sending the matrix in~(\ref{eqn_example1}) to $\det A$.

Clearly, $\Pi \setminus \Pi_L = \lbrace \alpha_3\rbrace$, therefore $\Xi_1 = \lbrace (\varpi_3, -\chi)\rbrace$.

Next, $L' \simeq \SL_3 \times \SL_3$, $L' \cap CK = L' \cap K \simeq \SL_3$ where $\SL_3$ is embedded in $\SL_3 \times \SL_3$ via $A \mapsto (A,(A^T)^{-1})$.
Then $L'/(L' \cap CK)$ is a symmetric space and
\[
\Xi'_2 = \lbrace (\varpi'_1 + \varpi'_4, 0), (\varpi'_2+\varpi'_5,0)\rbrace.
\]
Thus we get $\Xi_2 = \lbrace (\varpi_1 + \varpi_4, - \chi), (\varpi_2 + \varpi_5, - \chi) \rbrace$ and therefore $\Pi_{12} = \Pi$.

Clearly, $\mathfrak l = \mathfrak b_L + \mathfrak k$ and $B_S$ consists of all diagonal matrices in~$K$, therefore conditions~(\ref{S1})--(\ref{S6}) are satisfied with $S = B_S$, in which case $T \cap B_S = S$.
We now introduce characters $\psi_1,\psi_2,\psi_3 \in \mathfrak X(S)$ as follows.
For every
\[
t = \diag(t_1,t_2,t_3,t_1^{-1},t_2^{-1},t_3^{-1}) \in S
\]
and every $i = 1,2,3$, we put $\psi_i(t) = t_i$.
It is easy to see that the $S$-module $\mathfrak p_u / \mathfrak h_u$ is isomorphic to $\CC^1 \oplus \CC^1 \oplus \CC^1$ where the weights of the one-dimensional summands are $-\mu_1$, $-\mu_2$, and $-\mu_3$ with $\mu_1 = \psi_1 + \psi_2$, $\mu_2 = \psi_1 + \psi_3$, and $\mu_3 = \psi_2 + \psi_3$.
It follows that $\Xi'_3 = \lbrace \mu_1, \mu_2, \mu_3 \rbrace$.
We extend $\mu_1,\mu_2,\mu_3$ to $\mathfrak X(T)$ as $\widetilde \mu_1 = \varpi_2$, $\widetilde \mu_2 = \varpi_1 + \varpi_5$, $\widetilde \mu_3 = \varpi_4$.
Thus $\Xi_3 = \lbrace (\lambda_1,\chi_1), (\lambda_2,\chi_2), (\lambda_3,\chi_3)\rbrace$ with
\begin{equation*}
\begin{split}
(\lambda_1,\chi_1) &= (\varpi_2, 0) + a_1(\varpi_3,-\chi) + b_1(\varpi_1+\varpi_4,-\chi) + c_1(\varpi_2 + \varpi_5,-\chi);\\
(\lambda_2,\chi_2) &= (\varpi_1+ \varpi_5, 0) + a_2(\varpi_3,-\chi) + b_2(\varpi_1+\varpi_4,-\chi) + c_2(\varpi_2 + \varpi_5,-\chi);\\
(\lambda_3,\chi_3) &= (\varpi_4, 0) + a_3(\varpi_3,-\chi) + b_3(\varpi_1+\varpi_4,-\chi) + c_3(\varpi_2 + \varpi_5,-\chi),
\end{split}
\end{equation*}
and it remains to determine the coefficients $a_i,b_i,c_i$ for $i=1,2,3$.

We have $\Ker \iota = \ZZ\lbrace \varpi_1 - \varpi_3 + \varpi_4, \varpi_2 - \varpi_3 + \varpi_5 \rbrace$, hence by Theorem~\ref{thm_weight_lattice}
\[
\Lambda = \Ker \iota + \ZZ \lbrace \widetilde \mu_1, \widetilde \mu_2, \widetilde \mu_3 \rbrace = \ZZ \lbrace \varpi_2,\varpi_4, \varpi_1+\varpi_3, \varpi_1+\varpi_5,\varpi_3+\varpi_5 \rbrace.
\]
In particular, we find that $\Pi \subset \Lambda$.

For every $i=1,2,3$, put $\rho_i = \rho_{\mu_i}$.
The following table shows the values $\rho_i(\alpha_j)$:
\begin{center}
\begin{tabular}{|c|c|c|c|c|c|}
\hline

 & $\alpha_1$ & $\alpha_2$ & $\alpha_3$ & $\alpha_4$ & $\alpha_5$\\

\hline

$\rho_1$ & $0$ & $1$ & $0$ & $0$ & $-1$ \\

\hline

$\rho_2$ & $1$ & $-1$ & $1$ & $-1$ & $1$ \\

\hline

$\rho_3$ & $-1$ & $0$ & $0$ & $1$ & $0$ \\

\hline
\end{tabular}
\end{center}
We now compute the set~$\Sigma$.
Although the results of~\cite{Avd_degen} do not apply directly, we shall make use of some ideas presented in loc. cit.
First, \cite[Proposition~3.23]{Avd_degen} shows that $H$ has a finite index in its normalizer in~$G$, which implies $|\Sigma| = \rk \Lambda = 5$.
Next, let $\mathfrak h_0 \subset \mathfrak g$ be a Lie algebra obtained from~$\mathfrak h$ by degeneration via a one-parameter subgroup of~$G$ (here both $\mathfrak h$ and $\mathfrak h_0$ are regarded as points of the corresponding Grassmannian of~$\mathfrak g$) and let $N_0$ be the connected component of the identity of the normalizer of~$\mathfrak h_0$ in~$G$.
Then the discussion in~\cite[\S\S\,3.8--3.9]{Avd_degen} implies that each spherical root of $G/N_0$ is proportional to a spherical root of $G/H$.
To implement the above idea we consider the one-parameter subgroup $\phi \colon \CC^\times \to G$ given by $\phi(s) = \diag (s,1,s^{-1},s,1,s^{-1})$ and for every $s \in \CC^\times$ put $\mathfrak h_s = \phi(s)\mathfrak h$.
Then it is easy to see that $\mathfrak h_s$ consists of all matrices in $\mathfrak g$ of the form
\[
\begin{pmatrix}
a_1 & s^2a_2 & s^4a_3 & 0 & 0 & 0 \\
a_4 & a_5 & s^2a_6 & 0 & 0 & 0 \\
a_7 & a_8 & a_9 & 0 & 0 & 0 \\
y_1 & s^2y_2 & s^4y_3 & -a_1 & -s^2a_4 & -s^4a_7 \\
y_2 & y_4 & s^2y_5 & -a_2 & -a_5 & -s^2a_8 \\
y_3 & y_5 & y_6 & -a_3 & -a_6 & -a_9
\end{pmatrix}
\]
and the limit $\mathfrak h_0 = \lim \limits_{s \to 0} \mathfrak h_s$ consists of all such matrices with $s = 0$.
Applying again \cite[Proposition~3.23]{Avd_degen} we find that the subgroup $N_0$ consists of all matrices in~$G$ of the form
\[
\begin{pmatrix}
* & 0 & 0 & 0 & 0 & 0 \\
* & * & 0 & 0 & 0 & 0 \\
* & * & * & 0 & 0 & 0 \\
* & 0 & 0 & * & 0 & 0 \\
* & * & 0 & * & * & 0 \\
* & * & * & * & * & *
\end{pmatrix}.
\]
We see that $N_0$ is strongly solvable, hence \cite[Theorem~5.28(c)]{Avd_solv_inv} yields $\Sigma_G(G/N_0) = \lbrace \alpha_2,\alpha_3,\alpha_4 \rbrace$.
It immediately follows that $\lbrace \alpha_2,\alpha_3,\alpha_4 \rbrace \subset \Sigma$.
To determine the remaining two spherical roots in~$\Sigma$, we apply another reasoning.
Put $H_1 = K P_u$, then $G/H_1$ is parabolically induced from~$L/K$, which implies $\Sigma_G(G/H_1) = \Sigma_L(L/K)$ by~\cite[Proposition~20.13]{Tim}.
As $L/K \simeq L'/(L' \cap K)$, we obtain $\Sigma_G(G/H_1) = \lbrace \alpha_1 + \alpha_4, \alpha_2 + \alpha_5 \rbrace$.
Since $H \subset H_1$, Proposition~\ref{prop_two_SS} along with the list of all possible spherical roots (see, for instance, \cite[Table~1]{Avd_solv_inv}) yield the following two possibilities: either $\alpha_1+\alpha_4 \in \Sigma$ or $\alpha_1, \alpha_4 \in \Sigma$.
But the former one is impossible since $\alpha_4 \in \Sigma$ and $(\alpha_1 + \alpha_4, \alpha_4) > 0$, therefore $\alpha_1 \in \Sigma$.
Similarly, $\alpha_5 \in \Sigma$ and we finally obtain $\Sigma = \Pi$.

Knowing~$\Pi \cap \Sigma$, from Theorem~\ref{thm_gen_3}(\ref{thm_gen_3_b}) we deduce that $a_1=b_1=c_1 = 0$, $a_2 = 1$, $b_2=c_2 = 0$, and $a_3=b_3=c_3=0$.
Thus
\[
\Xi = \lbrace (\varpi_3,-\chi), (\varpi_1+\varpi_4,-\chi), (\varpi_2+\varpi_5,-\chi), (\varpi_2,0),(\varpi_1+\varpi_3+\varpi_5,-\chi),(\varpi_4,0) \rbrace.
\]
\end{example}

\begin{example} \label{ex_SL3}
Let $G = \SL_3$, then $\Pi = \lbrace \alpha_1,\alpha_2 \rbrace$ where $\alpha_i(t) = t_it_{i+1}^{-1}$ for all $i = 1,2$ and $t = \diag(t_1,t_2,t_3) \in T$.
Let $H \subset G$ be the subgroup of all matrices in~$G$ of the form
\[
\begin{pmatrix}
* & 0 & 0 \\
0 & * & 0 \\
0 & * & *
\end{pmatrix}
\]
and put $K = T$.
There are two different choices of a parabolic subgroup~$P \supset B^-$ such that $H$ is regularly embedded in~$P$, and we first fix the choice under which $P$, $L$, $C$ consist of all matrices in~$G$ of the form
\[
\begin{pmatrix}
* & * & 0 \\
* & * & 0 \\
* & * & *
\end{pmatrix}, \
\begin{pmatrix}
* & * & 0 \\
* & * & 0 \\
0 & 0 & *
\end{pmatrix}, \
\begin{pmatrix}
c & 0 & 0 \\
0 & c & 0 \\
0 & 0 & c^{-2}
\end{pmatrix},
\]
respectively.
Note that $CK =K$.
Clearly, $\mathfrak p_u$ and $\mathfrak h_u$ are given by all matrices of the form
\[
\begin{pmatrix}
0 & 0 & 0 \\
0 & 0 & 0 \\
* & * & 0
\end{pmatrix}
\ \text{and} \
\begin{pmatrix}
0 & 0 & 0 \\
0 & 0 & 0 \\
0 & * & 0
\end{pmatrix},
\]
respectively.

Since $\Pi \setminus \Pi_L = \lbrace \alpha_2 \rbrace$, we have $\Xi_1 = \lbrace (\varpi_2,-\varpi_2) \rbrace$.

Next, $L' \simeq \SL_2$, $L' \cap CK = L' \cap K \simeq \CC^\times$ where $\CC^\times$ is a maximal torus of~$\SL_2$, and we have $\Xi'_2 = \lbrace (\varpi_1', \chi), (\varpi_1', - \chi) \rbrace$ where $\chi$ is a generating character of~$\mathfrak X(L' \cap K)$.
Then $\Xi_2 = \lbrace (\varpi_1, -\varpi_1), (\varpi_1, \varpi_1-\varpi_2) \rbrace$ and therefore $\Pi_{12}= \Pi$.

Note that $B_LK$ is not open in~$L$; therefore, to compute the set~$\Xi_3$, we replace the subgroup $H$ with $H_1 = gHg^{-1}$ where \[
g = \begin{pmatrix} 1 & 0 & 0 \\ 1 & 1 & 0 \\ 0 & 0 & 1 \end{pmatrix}
\]
and put $K_1 = gKg^{-1}$.
Then $H_1$, $K_1$, and $(\mathfrak h_1)_u$ consist of all matrices of the form
\[
\begin{pmatrix}
t_1 & 0 & 0 \\
t_1-t_2 & t_2 & 0 \\
-y & y & t_3
\end{pmatrix}, \
\begin{pmatrix}
t_1 & 0 & 0 \\
t_1-t_2 & t_2 & 0 \\
0 & 0 & t_3
\end{pmatrix}, \
\begin{pmatrix}
0 & 0 & 0 \\
0 & 0 & 0 \\
-z & z & 0
\end{pmatrix},
\]
respectively, with $t_1t_2t_3 \ne 0$.
Observe that $B_LK_1$ is open in~$L$ and the group $B_S = K_1 \cap B_L$ equals~$C$.
Then conditions (\ref{S1})--(\ref{S6}) are satisfied by $K_1$ and~$S = B_S = C$, in which case $\Ker \iota = \ZZ\lbrace \alpha_1 \rbrace$.
Clearly, $\mathfrak p_u/(\mathfrak h_1)_u$ is a one-dimensional $S$-module of weight~$-\mu$ where $\mu = \iota(\alpha_2)$, hence $\Xi'_3 = \lbrace \mu \rbrace$ and we can take $\widetilde \mu = \alpha_2$.
Thus $\Lambda = \Ker \iota + \ZZ \widetilde \mu = \ZZ\Pi$.
In particular, we find that $\Pi_{12} \subset \Lambda$.

We have $\Xi_3 = \lbrace (\lambda, \chi) \rbrace$ where
\[
(\lambda, \chi) = (2\varpi_2-\varpi_1, 0) + a(\varpi_2,-\varpi_2) + b(\varpi_1,-\varpi_1) + c(\varpi_1,\varpi_1-\varpi_2).
\]
Next, $\rho_\mu(\alpha_1) = 0$ and $\rho_\mu(\alpha_2) = 1$.
As $H$ is strongly solvable, from \cite[Theorem~5.28(c)]{Avd_solv_inv} we find that $\Sigma = \Pi$.
Then Theorem~\ref{thm_gen_3}(\ref{thm_gen_3_b}) yields $a=-1$ and $b+c=1$, which does not determine the coefficients uniquely.
The reason is that $K$ is contained in a proper parabolic subgroup of~$L$.

The second choice of the parabolic subgroup~$P$ is $P = B^-$, in which case an application of Theorem~\ref{thm_AG} yields
\[
\Xi = \lbrace (\varpi_1,-\varpi_1), (\varpi_2,-\varpi_2), (\varpi_1, \varpi_1-\varpi_2), (\varpi_2,\varpi_1) \rbrace.
\]
\end{example}

The next example is based on~\cite[Example~4.9]{Avd_degen}.

\begin{example} \label{ex_SO7}
Let $G = \SO_7$ preserving the symmetric bilinear form on~$\CC^7$ whose matrix has ones on the antidiagonal and zeros elsewhere.
Then the Lie algebra~$\mathfrak g$ consists of all $(7\times7)$-matrices that are skew-symmetric with respect to the antidiagonal.
We have $\Pi = \lbrace \alpha_1, \alpha_2, \alpha_3 \rbrace$ where $\alpha_1(t)=t_1t_{2}^{-1}$, $\alpha_2(t)=t_2t_{3}^{-1}$, $\alpha_3(t) = t_3$ for all $t = \diag(t_1,t_2,t_3,1,t_3^{-1},t_2^{-1},t_1^{-1}) \in T$.
Consider a connected subgroup $H \subset G$ regularly embedded in a parabolic subgroup $P \supset B^-$ such that the Lie algebras $\mathfrak h$ and $\mathfrak p$ consist of all matrices in $\mathfrak g$ having the form
\[
\begin{pmatrix}
x+y & 0 & 0 & 0 & 0 & 0 & 0\\
a & x & * & 0 & 0 & 0 & 0\\
b & * & y & 0 & 0 & 0 & 0\\
2c & -b & a & 0 & 0 & 0 & 0\\
* & c & 0 & -a & -y & * & 0\\
* & 0 & -c & b & * & -x & 0\\
0 & * & * & -2c & -b & -a & -x-y\\
\end{pmatrix}
\ \text{and} \
\begin{pmatrix}
* & 0 & 0 & 0 & 0 & 0 & 0\\
* & * & * & 0 & 0 & 0 & 0\\
* & * & * & 0 & 0 & 0 & 0\\
* & * & * & 0 & 0 & 0 & 0\\
* & * & 0 & * & * & * & 0\\
* & 0 & * & * & * & * & 0\\
0 & * & * & * & * & * & *\\
\end{pmatrix},
\]
respectively.
Then $L$ consists of all block-diagonal matrices with the sequence of blocks $(s, A, 1, (A^\natural)^{-1}, s^{-1})$ where $s \in \CC^\times$, $A \in \GL_2$, and $A^\natural$ stands for the transpose of $A$ with respect to the antidiagonal;
$K$ consists of all matrices in $L$ satisfying $s = \det A$;
and $C$ consists of all diagonal matrices of the form $\diag(t_1,t_2,t_2,1,t_2^{-1}, t_2^{-1},t_1^{-1})$.
Note that $K \simeq \GL_2$ and there is a $\GL_2$-module isomorphism $\mathfrak p_u/\mathfrak h_u \simeq \CC^2 \oplus \CC^1$ where the latter module is acted on by $\GL_2$ via $(A, (x,y)) \mapsto ((\det A)^{-1}Ax, (\det A)^{-1}y)$.

In what follows, for every subgroup $F \subset G$ we denote by $\widetilde F$ the preimage of~$F$ in $\widetilde G = \Spin_7$ and regard each character of~$F$ as a character of~$\widetilde F$.
The sets $\Xi_1,\Xi_2,\Xi_3$ will be considered for the homogeneous space $\widetilde G / \widetilde H$ (isomorphic to $G/H$ as a $\widetilde G$-variety).

Let $\psi_1$ (resp.~$\psi_3$) be the restriction of $\varpi_1$ (resp.~$\varpi_3$) to~$C$.

Clearly, $\Pi \setminus \Pi_L = \lbrace \alpha_2 \rbrace$, which yields $\Xi_1 = \lbrace (\varpi_1,-\psi_1), (\varpi_3,-\psi_3) \rbrace$.

Next observe that $CK = L$, therefore $\Xi_2 = \varnothing$ and hence $\Pi_{12} = \lbrace \alpha_1, \alpha_3 \rbrace$.

We automatically have $S = K$, and so $\Ker \iota = \ZZ\lbrace \alpha_1 - \alpha_3 \rbrace$.
The $S$-module $\mathfrak p_u / \mathfrak h_u$ is the direct sum of two simple $S$-modules: the first one of dimension~$2$ with lowest weight $-\mu_1$ and the second one of dimension~$1$ with lowest weight $-\mu_2$ where $\mu_1 = \iota(\alpha_1+\alpha_2)$ and $\mu_2 = \iota(\alpha_1+\alpha_2+\alpha_3)$.
It follows that $\Xi'_3 = \lbrace \mu_1, \mu_2 \rbrace$ and we can take $\widetilde \mu_1 = \alpha_1+\alpha_2$ and $\widetilde \mu_2 = \alpha_1 + \alpha_2 + \alpha_3$.
Thus $\Lambda = \Ker \iota + \ZZ \lbrace \widetilde \mu_1, \widetilde \mu_2\rbrace = \ZZ\Pi$.
In particular, we find that $\Pi_{12} \subset \Lambda$.

We have $\Xi_3 = \lbrace (\lambda_1,\chi_1), (\lambda_2,\chi_2) \rbrace$ where
\begin{equation*}
\begin{split}
(\lambda_1,\chi_1) &= (\varpi_1+\varpi_2-2\varpi_3, 0) + a_1(\varpi_1, - \psi_1) + b_1(\varpi_3,-\psi_3);\\
(\lambda_2,\chi_2) &= (\varpi_1,0) + a_2(\varpi_1,-\psi_1) + b_2(\varpi_3,-\psi_3).
\end{split}
\end{equation*}

The following table shows the values $\rho_i(\alpha_j)$:
\begin{center}
\begin{tabular}{|c|c|c|c|}
\hline

 & $\alpha_1$ & $\alpha_2$ & $\alpha_3$ \\

\hline

$\rho_1$ & $-1$ & $2$ & $-1$ \\

\hline

$\rho_2$ & $1$ & $-1$ & $1$ \\

\hline
\end{tabular}
\end{center}

As $L' \subset K \subset L$, we can apply algorithms developed in~\cite{Avd_degen} and compute $\Sigma = \lbrace \alpha_1 + \alpha_2, \alpha_2 + \alpha_3, \alpha_3 \rbrace$, which implies $\Pi_{12} \cap \Sigma = \lbrace \alpha_3 \rbrace$.
Then we deduce from Theorem~\ref{thm_gen_3}(\ref{thm_gen_3_b}) that $a_1 = -1$, $b_1 = 2$, $a_2 = -1$, $b_2 = 1$ and finally obtain
\[
\Xi = \lbrace (\varpi_1,-\psi_1), (\varpi_3, -\psi_3), (\varpi_2, \psi_1-2\psi_3), (\varpi_3,\psi_1-\psi_3)\rbrace.
\]

Note that for $\alpha_1 \in \Pi_{12}$ the necessary conditions of Propositions~\ref{prop_alpha_in_Sigma1} and~\ref{prop_alpha_in_Sigma2}(\ref{prop_alpha_in_Sigma2_a}) are fulfilled but $\alpha_1 \notin \Sigma$.
Thus the above-mentioned conditions are not sufficient in general.
\end{example}

\section{Concluding remarks}
\label{sect_concluding_remarks}

Retain the notation of~\S\S\,\ref{sect_main_results}--\ref{sect_proofs}.
In this section, we discuss several directions of further research related to the problem of computing the set $\Pi_{12} \cap \Sigma$ purely in terms of $P$ and~$H$.
Let $\alpha \in \Pi_{12}$.

If $\alpha \in \Pi_L$ then $P_\alpha \cap L \supset B_L$ is a minimal parabolic subgroup of~$L$, therefore by Proposition~\ref{prop_localization}(\ref{prop_localization_a}) and Lemma~\ref{lemma_unstable} the condition $\alpha \in \Sigma$ is equivalent to $\alpha \in \Sigma_L(P/H)$.
Thus a natural problem is to develop methods for computing the set $\Sigma_L(P/H)$.
By~\cite[Theorem~3.7]{Tim}, $P/H$ is a smooth affine spherical $L$-variety.
Up to certain reductions, a complete classification of all smooth affine spherical varieties was obtained by Knop and Van Steirteghem in~\cite{KnVS}, and it is natural and seemingly feasible to compute all relevant invariants (including the weight monoid and spherical roots) for all items in their list.

If $\alpha \in \Pi \setminus \Pi_L$ then, again by Proposition~\ref{prop_localization}(\ref{prop_localization_a}) and Lemma~\ref{lemma_unstable} (see also Lemma~\ref{lemma_alpha_in_Sigma}), the condition $\alpha \in \Sigma$ holds if and only if $U_{-\alpha}$ moves an element in~$\mathcal E$ if and only if $U_{-\alpha}$ moves an element in~$\mathcal F$.
Now observe that $U_{-\alpha}$ is normalized by $B_L$ and hence by~$B_S$, therefore we fall into the following situation: given a connected reductive group $N$, an affine spherical $N$-variety $X$ is equipped with an action of a one-dimensional unipotent group~$U$ normalized by a Borel subgroup $B_N \subset N$, and the question is whether $U$ moves a $B_N$-stable prime divisor in~$X$.
In our situation, $N = L$, $X = P/H$, $B_N = B_L$, $U = U_{-\alpha}$ or $N = S^0$, $X = P_u/H_u$, $B_N = B_S^0$, $U = U_{-\alpha}$.
When $N = B_N$ is a torus and the action of $U$ on~$X$ is nontrivial then the following statements are known to be true:
\begin{enumerate}[label=\textup{(\arabic*)},ref=\textup{\arabic*}]
\item \label{toric1}
\cite[Theorem~2.7 and Corollary~2.8]{Lie} the image of the Lie algebra of $U$ in the derivations of $\CC[X]$ is uniquely determined by the $B_N$-weight of~$U$;

\item \label{toric2}
\cite[Lemma~2.2]{AZK} $U$ moves a $B_N$-stable prime divisor in~$X$.
\end{enumerate}
For the action of $U_{-\alpha}$ on $P_u/H_u$, we have a criterion of nontriviality given by Lemma~\ref{lemma_acts_nontrivially}.
Unfortunately, in Example~\ref{ex_SO7} the groups $U_{-\alpha_1}, U_{-\alpha_3}$ both act nontrivially on $P_u/H_u$ and have the same $B_S$-weight $\iota(\alpha_1)=\iota(\alpha_3)$; however, by Lemma~\ref{lemma_alpha_in_Sigma}, $U_{-\alpha_3}$ moves an element in~$\mathcal F$ and $U_{-\alpha_1}$ does not, which means that in the general case both properties~(\ref{toric1}) and~(\ref{toric2}) may fail.
Thus the question of existence of a $U_{-\alpha}$-unstable element in~$P_u/H_u$ requires either a deeper study of $B_N$-normalized actions of one-dimensional unipotent groups on affine spherical $N$-varieties or another approach.

%%%%%%%%%%%%%%%%%%%%%%%%%%%%%%%%%%%%%%%%%%%%%%%%%%%%%%%%%%%%%%%%%%%%%%%%%%%%%%%%%%%%%%%%%%

\end{document}